\documentclass[a4paper, 11pt]{amsart}
\usepackage{mathrsfs}
\usepackage{amssymb}
\usepackage[usenames,dvipsnames,svgnames,table]{xcolor}

\usepackage[centertags]{amsmath}
\usepackage{amsfonts}
\usepackage{amsthm}
\usepackage{times}

\usepackage{enumerate, enumitem}

\usepackage{hyperref}
\usepackage[capitalise, noabbrev]{cleveref}

\hypersetup{
    colorlinks,
    linkcolor={RoyalBlue},
    citecolor={OliveGreen},
    urlcolor={blue!80!black}
}

\usepackage{enumerate} %
\usepackage{xcolor}
\usepackage[margin=1.25in]{geometry}
\usepackage{centernot}
\usepackage{csquotes}
\usepackage[symbol]{footmisc}
\usepackage{todonotes}
\usepackage{tabularx}

\renewenvironment{enumerate}{\begin{enumorig}[label=\textup{(\roman*)}, noitemsep, 
topsep=2pt plus 2pt, labelindent=.2em, leftmargin=*, widest=iii]}{\end{enumorig}}

\newenvironment{enumerateNumF}{\begin{enumorig}[label=\textup{(f\arabic*)}, noitemsep, 
topsep=2pt plus 2pt, labelindent=.2em, leftmargin=*, widest=f88]}{\end{enumorig}}

\newenvironment{enumerateNumSF}{\begin{enumorig}[label=\textup{(sf\arabic*)}, noitemsep, 
topsep=2pt plus 2pt, labelindent=.2em, leftmargin=*, widest=sf88]}{\end{enumorig}}

\newenvironment{enumerateNumO}{\begin{enumorig}[label=\textup{(o\arabic*)}, noitemsep, 
topsep=2pt plus 2pt, labelindent=.2em, leftmargin=*, widest=o88]}{\end{enumorig}}

\newenvironment{enumerateNumR}{\begin{enumorig}[label=\textup{(r\arabic*)}, noitemsep, 
topsep=2pt plus 2pt, labelindent=.2em, leftmargin=*, widest=r88]}{\end{enumorig}}

\newenvironment{enumerateNumL}{\begin{enumorig}[label=\textup{(l\arabic*)}, noitemsep, 
topsep=2pt plus 2pt, labelindent=.2em, leftmargin=*, widest=l8]}{\end{enumorig}}

\newtheorem{theorem}{Theorem}
\newtheorem*{theorem*}{Theorem}

\newtheorem{corollary}[theorem]{Corollary}
\newtheorem{lemma}[theorem]{Lemma}
\newtheorem{obs}[theorem]{Observation}

\newtheorem{problem}{Problem}

\newcommand{\supp}{\mathrm{supp}\,}
\newcommand{\range}{\mathrm{range}}

\newcommand{\calF}{\mathcal{F}}
\newcommand{\calG}{\mathcal{G}}

\newcommand{\calO}{\mathcal{O}}
\newcommand{\calS}{\mathcal{S}}

\newcommand{\Oh}{\mathcal{O}} 
\newcommand{\N}{\mathbb{N}}

\newcommand{\R}{\mathbb{R}}
\newcommand{\Z}{\mathbb{Z}}

\newcommand{\full}{\mathrm{full}}
\newcommand{\depth}{\mathrm{depth}}
\newcommand{\spn}{\mathrm{span}\,}

\newcommand{\hj}{\widehat{j}}
\newcommand{\mrm}{\mathrm}

\DeclareMathOperator{\Sf}{\mathcal{S}_{\varphi}}

\let\leq\leqslant
\let\geq\geqslant
\let\subset\subseteq

\usepackage{xparse}

\begin{document}

\title{The depth of Tsirelson's norm}

\author[K.~Beanland]{Kevin Beanland}
\address[K.~Beanland]{Department of Mathematics, Washington and Lee University, Lexington, VA 24450.}
\email{beanlandk@wlu.edu}

\author[J.~Hodor]{J\c{e}drzej Hodor}
\address[J.~Hodor]{Theoretical Computer Science Department\\ 
  Jagiellonian University\\ 
  Kraków, Poland}
\email{jedrzej.hodor@gmail.com}

	\thanks{2010 \textit{Mathematics Subject Classification}. Primary: }
	
	\thanks{J.\ Hodor is partially supported by a Polish National Science Center grant (BEETHOVEN; UMO-2018/31/G/ST1/03718).}

	
\begin{abstract}
    Tsirelson's norm $\|\cdot \|_T$ on $c_{00}$ is defined as the supremum over a certain collection of iteratively defined, monotone increasing norms $\|\cdot \|_k$. 
    For each positive integer $n$, the value $j(n)$ is the least integer $k$ such that for all $x \in \R^n$ (here $\R^n$ is considered as a subspace of $c_{00}$), $\|x\|_T = \|x\|_k$.
    In 1989 Casazza and Shura \cite{CS-book} asked what is the order of magnitude of $j(n)$. 
    It is known that $j(n) \in \Oh(\sqrt{n})$ \cite{Beanland_2018}.
    We show that this bound is tight, that is, $j(n) \in \Omega(\sqrt{n})$.
    Moreover, we compute the tight order of magnitude for some norms being modifications of the original Tsirelson's norm.
    \end{abstract}

\renewcommand\contentsname{}

\maketitle

\tableofcontents


        \section{Introduction}
    \label{sec:introduction}
    In 1974 Tsirelson \cite{Tsirel_son_1974} constructed a reflexive Banach space containing no embedding of $c_0$ or $\ell_p$ for each $1 \leq p < \infty$. 
The idea evolved throughout the years, and what is nowadays called Tsirelson's space is the dual of the original space, usually presented according to the description given in \cite{FJ-Tsirelson}. 
Tsirelson's space not only served as a counterexample in Banach space theory but the inductive process Tsirelson developed for defining the norm eventually lead to many breakthroughs in several areas of mathematics.
See Tsirelson's webpage for an exhaustive list of publications concerning Tsirelson's space up to 2004 \cite{Ts-webpage}, we refer directly to some of the most notable ones \cite{OSc-Acta,GoMa-JAMS,Baudier_2018}.
We also refer to a monograph of Casazza and Shura on Tsirelson's space \cite{CS-book}.

Tsirelson's space is the completion of $c_{00}$ under a certain norm -- we call it Tsirelson's norm. Let $\calS_1 := \{F \subset \N : |F| \leq \min F\} $ be the Schreier family \cite{Schreier}. We start by defining $\|\cdot\|_0$ as the supremum norm on $c_{00}$. Next, for each non-negative integer $m$ and for each $x \in c_{00}$ we define
    \[\|x\|_{m+1} := \max\left\{\|x\|_{m}, \sup \left\{ \frac{1}{2} \sum_{i}^d \|E_i x\|_{m} : E_1 < \dots < E_d, \{\min E_i : i \in [d]\} \in \calS_1 \right\} \right\}.\]
For subsets of integers $E,E'$ and $x \in c_{00}$, by $Ex$ we mean the coordinatewise multiplication of $x$ and the characteristic function of $E$, and by $E < E'$ we mean $\max E < \min E'$. See the next section for a more careful definition, however, generalized and stated in a slightly different spirit. Tsirelson's norm $\|x\|_{T}$ is defined as the supremum over $\|x\|_m$ for all non-negative integers $m$.

The above definition can be described more intuitively as a combinatorial game. A vector $x \in c_{00}$ is provided on input and the goal is to maximize the result. We start with the supremum norm. Then, in each step, we either take the current result or split the vector in some way dependent on the family $\calS_1$.
However, if we choose to split, we must pay a penalty of multiplying the current result by $\frac{1}{2}$.
Next, we proceed with the same game on each part of the split, summing the results afterward.

It is not hard to see that for every $x \in c_{00}$ there exists a positive integer number $M$ such that the norm stabilizes starting from the $M$th step, that is, $\|x\|_M = \|x\|_{M+1} = \dots = \|x\|_{T}$. We denote such minimal $M$ as $j(x)$. Now, for a positive integer $n$, we define $j(n)$ to be the maximum value of $j(x)$ over all $x \in c_{00}$ such that $x_{n+1} = x_{n+2} = \dots = 0$. The function $j(n)$ measures the complexity of computing Tsirelson's norm for finite vectors. In the game interpretation, $j(n)$ is the value of the longest optimal strategy for an input of length $n$.

This concept was introduced and initially studied in 1989 by Casazza and Shura \cite{CS-book}. 
They proved that $j(n) \in \Oh(n)$, and asked for the exact order of magnitude of $j(n)$.
In 2017, Beanland, Duncan, Holt, and Quigley \cite[Theorem~3.17]{BDHQ} provided the first non-trivial lower bound, namely, they showed that $j(n) \in \Omega(\log n)$.
One year later, Beanland, Duncan, and Holt \cite{Beanland_2018} proved that $j(n) \leq \calO(\sqrt{n})$. 
In this paper, we finally resolve the question of Casazza and Shura, by proving that $\Omega(\sqrt{n}) \leq j(n)$ (see \cref{cor:classic-schreier}). Combining the two results, we obtain the following.
\begin{theorem}\label{th:main-Schreier}
    For every positive integer $n$ we have
        \[\sqrt{2n} - 3 \leq j(n) \leq 2\sqrt{n} + 4.\]
\end{theorem}

In the remaining part of the introduction, we discuss some natural modifications of Tsirelson's norm.
From a modern point of view, the choice of $\frac{1}{2}$ and $\calS_1$ in the definition of Tsirelson's norm may seem a little bit artificial. 
Indeed, one can replace $\frac{1}{2}$ by any real number $0 < \theta < 1$ and $\calS_1$ by any regular family $\calF$ to obtain a norm $\|\cdot \|_{T[\theta,\calF]}$, and in turn a Banach space $T[\theta,\calF]$ -- see e.g.\ \mbox{\cite[Chapter~1]{ATo-book}~or~\cite[Chapter~3]{ATol-Memoirs}}. 
The generalized Tsirelson's spaces have many interesting properties and they are connected to various branches of mathematics, e.g.\ to logic \cite{Bellenot_1984}.
The notion of regular families is a natural abstraction of the crucial properties of the Schreier family -- see the next section for the definition. 
For the generalized version of the norm, one can still define the function $j_{\theta,\calF}$ in an analogous way, and again ask for the order of magnitude. 
In this paper, we stick to $\theta = \frac{1}{2}$, and we will consider various examples of regular families $\calF$. Let us define them now.

First, for every $\varphi$ increasing and superadditive function on positive integers, we define
    \[\calS_\varphi := \{F \subset \N : |F| \leq \varphi(\min F)\}.\]
Note that $\calS_{\mrm{id}} = \calS_1$. 
This is a very natural generalization of the Schreier family. 
Some properties for some particular cases of $\varphi$ were studied in \cite{BCF21}. 
It is also worth noting that the collection of Banach spaces $T[\frac{1}{2},\Sf]$ gained the attention of the research community in connection with the meta-problem of so-called explicitly defined Banach spaces \cite{Go-blog, Khanaki_2021, casazza2022nondefinability}.

Let $k$ be a positive integer, we consider the family consisting of unions at most $k$ Schreier sets, namely, we define
    \[k\calS_1 := \{F \subset \N : \exists_{E_1,\dots,E_k\in \calS_1} F = \bigcup_{i=1}^k E_i\}.\]
Some properties of such families were studied in \cite{BGHH22}. Moreover, $k \calS_1$ can be seen as the so-called convolution of the family $\calS_1$ with the family consisting of all sets with at most $k$ elements. 
By convoluting regular families, one can produce many Banach spaces with interesting properties -- see~\cite{AlA-Dissertationes}~or~\cite[Chapter~2]{ATol-Memoirs}.
We consider two more examples of regular families constructed in a similar spirit.
The first one, denoted by $\calS_2$, is the convolution of the Schreier family with itself, and the second one is the convolution of the Schreier family with $\calS_2$.
We have
    \begin{align*}
        \calS_2 &:= \{F \subset \N : \exists_{E_1,\dots,E_\ell\in \calS_1} F = \bigcup_{i=1}^\ell E_i, \{\min E_i : i \in [\ell]\} \in \calS_1\},\\
        \calS_3 &:= \{F \subset \N : \exists_{E_1,\dots,E_\ell\in \calS_2} F = \bigcup_{i=1}^\ell E_i, \{\min E_i : i \in [\ell]\} \in \calS_1\}.
    \end{align*}
For each of the above families, we give the exact order of magnitude for the function $j_\calF(n)$. 
Note that in some cases we decided to present less technical proofs instead of obtaining better constants.

\begin{theorem}
    Let $\varphi$ be an increasing and superadditive function on positive integers. For each positive integer $n$, let $\varphi_\Sigma^{-1}(n) = \min\{\ell \in \Z : n \leq \sum_{i=1}^\ell \varphi(i)\}.$ We have
        \[j_{\Sf}(n) \in \Theta(\varphi_\Sigma^{-1}(n)).\]
    For all positive integers $n,k$, let $p_k(n) = n^k$ and $e_k(n) = k^n$, then
        \[j_{\calS_{p_k}}(n) \in \Theta(n^{\frac{1}{k+1}}) \ \ \ \mrm{and}  \ \ \ j_{\calS_{e_k}}(n) \in \Theta(\log_k n ).\]
    For a fixed positive integer $k$ with $k \geq 2$, we have
        \[ j_{k\calS_1}(n) \in \Theta(\log n). \]
    If $k$ is not fixed, then
        \[j_{k\calS_1}(n) \in \Theta(\frac{1}{k}\log n).\]
    Last but not least, we have
        \[j_{\calS_2}(n) \in \Theta(\sqrt{\log n}) \ \ \ \mrm{and}  \ \ \  j_{\calS_2}(n) \in \Theta(\sqrt{\log^* n}).\]
\end{theorem}
For convenience of the reader, we refer to the proof of each of the bounds in the following table.
\ \newline
\begin{center}
\begin{tabular}{ c|c|c } 
  & \ \ \ Lower bounds (\cref{sec:lower}) \ \ \ & \ \ \ Upper bounds (\cref{sec:upper}) \ \ \ \\ \hline
 $\calS_1$ & \cref{cor:classic-schreier} & \cite{Beanland_2018} or \cref{th:upper:S-y} \\ \hline
 $\Sf$ & \cref{cor:lower:Sf} & \cref{th:upper:S-y} \\ \hline
 $\calS_p,\calS_e$ & direct computation &  direct computation\\ \hline
 $k\calS_1$ & \cref{cor:lower:kS1} & \cref{thm:upper:kS_1} \\ \hline
 $\calS_2$ & \cref{cor:lower:S2} & \cref{thm:upper:S2} \\ \hline
 $\calS_3$ & \cref{cor:lower:S3} & \cref{thm:upper:S3}
\end{tabular}
\end{center}
\ \newline
By $\log^* n$ we mean the iterated logarithm of $n$, that is, how many times do we have to take the logarithm of $n$ until we reach a number below $1$. 
This function emerges often in computer science, and in particular in the field of analyzing the complexity of algorithms. 
For example, the average time complexity of a query in the classical Find-Union data structure is of order $\log^* n$. 
The iterated logarithm is an extremely slow-growing function, e.g.\ $\log^*(2^{65536}) = 5$.
Observe that one can define the families $\calS_4,\calS_5,\dots$ analogously to $\calS_2$ and $\calS_3$.
It is clear that the functions $j_{\calS_k}$ for $k \geq 4$ are even slower growing than $\log^*$.
This indicates that these functions do not even have natural names, therefore, to avoid unnecessary technicalities, we decided not to consider these families. 
However, we believe that our tools are sufficient to compute the order of magnitude functions of $j_{\calS_k}$ for any positive integer $k$.

The paper is organized as follows. In the next section, we settle the required notation.
In \cref{sec:fullness}, we discuss special members in regular families called \emph{full} sets that are very useful in proving our results.
In \cref{sec:tools_lower}, we introduce some abstract tools for the lower bounds, and in \cref{sec:lower}, we use the tools to establish lower bounds on the function $j_\calF(n)$ in the case of $\calF$ being one of the families that we are interested in.
Next, in \cref{sec:tools_upper}, we introduce abstract tools for the upper bounds, and in \cref{sec:upper}, we use the tools to establish upper bounds on the function $j_\calF(n)$ in the case of $\calF$ being one of the families that we are interested in.
Finally, in \cref{sec:open}, we discuss some related open problems.

 	\section{Preliminaries}
    \label{sec:preliminaries}
    Let $\N$ be the set of all positive integers, and let $\N_0 := \N \cup \{0\}$. 
For two integers $a,b$ with $a\leq b$ we write $[a,b]$ to denote the set $\{a,a+1,\dots,b\}$, if $a > b$, then $[a,b] := \emptyset$, and $[0] := \emptyset$. 
For a positive integer $a$, we abbreviate $[a] := [1,a]$. 
For any $E,F \subset \N$ the expression $E < F$ is a short form of writing that $\max E < \min F$, and similarly for $\leq,>,\geq$. 
An inequality between $E$ and some $a \in \N$ should be understood as an inequality between $E$ and $\{a\}$.
For every $E \subset \N$ and for all distinct $a,b \in E$ we say that $a,b$ are \emph{consecutive in $E$} if $[a+1,b-1] \cap E = \emptyset$. 
When we omit a base of a logarithm, we consider the base $2$.
Let $\tau$ be the power tower function, that is, for every real number $x$, we set $\tau(0,r) = r$, and for every $i \in \N$ we set $\tau(i,r) = 2^{\tau(i-1,r)}$. 
Let $\log^*$ be the iterated logarithm, that is, for every real number $r$, we have $\log^* r = \min \{i \in \N_0 : r \leq \tau(i,1)\}$.

For a vector $x \in \R^\N$ we write $\supp x := \{i \in \N : x_i \neq 0\}$. 
We write $c_{00}$ for all the vectors $x\in \R^\N$ such that $|\supp x| < \infty$. 
We write $c_{00}^+$ for all nonzero vectors $x \in c_{00}$ such that $x_i \geq 0$ for all $i \in \N$. 
For all $E \subset \N$ and $x \in c_{00}$ we write $Ex$ for the projection of $x$ onto $E$, that is, $(Ex)_i = x_i$ whenever $i \in E$ and $(Ex)_i = 0$ otherwise. 
For each $i \in \N$ we define $e_i \in c_{00}^+$ to be the vector with $(e_i)_i=1$ and $(e_i)_j = 0$ for each $j \in \N \backslash \{i\}$. 
For a linear functional $f:\R^\N \rightarrow \R$ we write $\supp f := \{i \in \N : f(e_i) \neq 0\}$. 
For each $i \in \N$, the functional $e_i^*: \R^\N \rightarrow \R$ is such that for each $x \in c_{00}$ we have $e^*_i(x) = x_i$. 

Let $\calF$ be a family of finite subsets of $\N$. 
We say that $\calF$ is \emph{hereditary} if for every $F \in \calF$ and $G\subset F$ we have $G \in \calF$. 
We say that $\calF$ is \emph{spreading} if for every $n \in \N$ and for all $\ell_1,\dots,\ell_n,k_1,\dots,k_n \in \N$ such that $\ell_i \leq k_i$ for all $i \in [n]$, and $\{\ell_1,\dots,\ell_n\} \in \calF$ we have $\{k_1,\dots,k_n\} \in \calF$. 
We say that $\calF$ is \emph{compact} if it is compact as a subset of $\{0,1\}^\N$ with the product topology under the natural identification.
Finally, we say that $\calF$ is \emph{regular} if it is hereditary, spreading, and compact. 
Perhaps, the most prominent example of a regular family is the Schreier family $\calS_1$ defined in the introduction. It is quite straightforward to check that all the families defined in the introduction are regular ($\calS_\varphi$, $k\calS_1$, $\calS_2$, and $\calS_3$).

Next, we proceed with introducing Tsirelson's norm. Fix a regular family $\calF$. We define
    \[W_0(\calF) := \{e_i^* : i \in \N\} \cup \{-e_i^* : i\in \N\}.\]
For each $m \in \N_0$ we define
    \begin{align*}
        W_{m+1}(\calF) &:= 
        W_m(\calF) \ \cup\\& \left\{\frac{1}{2} \sum_{i=1}^d f_i : \ \ f_i \in W_m(\calF), \ \ \{\min \supp f_i : i \in [d]\} \in \calF, \ \ \supp f_1 < \dots < \supp f_d\right\}.
    \end{align*}
We define the set of \emph{norming functionals for $\calF$},
    \[W(\calF) := \bigcup_{m=0}^\infty W_m(\calF).\]
As the name suggests, for all $x \in c_{00}$ and $m \in \N_0$ we define
    \[\|x\|_{\calF,m} := \sup\{f(x) : f \in W_m(\calF)\}.\]
And finally, for every $x \in c_{00}$ we define the $\calF$-Tsirelson's norm
    \[\|x\|_{\calF} := \sup\{\|x\|_{\calF,m} : m \in \N_0\}.\]
It is not hard to observe that the norm $\|\cdot \|_{\calS_1}$ coincide with the norm $\|\cdot \|_T$ defined in the introduction (see e.g.\ \cite[Chapter~1]{ATo-book}). 
The definition introduced in \cref{sec:introduction} is the classical definition, whereas the definition above is much more handy to work with.

For every $f \in W(\calF)$ the \emph{depth of $f$ with respect to the family $\calF$} is
    \[\mathrm{depth}_\calF(f) := \min\{m \in \N_0 : f \in W_m(\calF)\}\]
Let $f \in W(\calF)$ and suppose that $\depth_\calF(f) = m + 1$ for some $m\in  \N_0$. 
By definition, there exist $f_1,\dots,f_d \in W_m(\calF)$ such that $f = \frac{1}{2}\sum_{i=1}^d f_i$, the set $\{ \min \supp f_i : i \in [d]\}$ is in $\calF$, and $\supp f_1 < \dots < \supp f_d$. 
Note that in general, $f_1,\dots,f_d$ are not uniquely determined. 
We say that $f_1,\dots,f_d$ are \emph{$\calF$-building for $f$}.
For every $x \in c_{00}$ we define
    \[j_\calF(x) := \min \{\mathrm{depth}_\calF(f) : \ \ f \in W(\calF), \ \ f(x) = \|x\|_\calF\}.\]
For all $a,b \in \N$ with $a \leq b$ we define
    \[j_\calF(a,b) := \max \{j_\calF(x) : \ \ x\in c_{00}, \ \ \supp x \subset [a,b]\}.\]
It is not difficult to see that in the above definition, $c_{00}$ can be replaced with $c_{00}^+$ without changing any value of $j_\calF(a,b)$. We will use this fact implicitly sometimes. Finally, for each $n \in \N$ we define
    \[j_\calF(n) := j_\calF(1,n).\]

We will need the following simple observation on the behavior of the function $j_\calF$.
\begin{obs}\label{obs:j-ab:ineq}
Let $\calF$ be a regular family, and let $a,b,c,d \in \N$ such that $[a,b] \subset [c,d]$. 
We have $j_\calF(a,b) \leq j_\calF(c,d)$.
 \end{obs}
 \begin{proof}
     Let $x \in c_{00}$ with $\supp x \subset [a,b]$, and such that $j_\calF(a,b) = j_\calF(x)$. Then, $\supp x \subset [c,d]$, thus $j_\calF(a,b) = j_\calF(x) \leq j_\calF(c,d)$.
\end{proof}
    
	\section{Full sets in regular families}
    \label{sec:fullness}
    Given a positive integer, e.g.\ $a=10$, and a regular family $\calF$, starting with $\{a\}$, one can greedily add consecutive integers to the set to determine the threshold after which the set is no longer a member of $\calF$.
Say that $\calF= \calS_1$. 
It is clear that $\{10,11,12,\dots,19\}$ is still in $\calS_1$, however, $\{10,11,12,\dots,19,20\}$ is not.
On the other hand, if $\calF = 2\calS_1$, then not only $\{10,11,12,\dots,19,20\}$ is in $2\calS_1$, but even $\{10,11,12,\dots,19,20,21,\dots,38,39\}$ is in $2\calS_1$. 
The threshold varies a lot among regular families.
We find it useful to formalize this notion as follows.
For each regular family $\calF$, and for each $a \in \N$, let
    \[
        \range_\calF(a) := \max \{m \in \N : \ \ [a,m-1] \in \calF\}.
    \]
First, note the following trivial observation.
\begin{obs} \label{obs:r-forces} Let $F \subset \N$. For every regular family $\calF$ if $\max F < \range_\calF(\min F)$, then $F \in \calF$.
\end{obs}
Next, observe that for some of the families that we consider it is very easy to compute the value of $\range_\calF$.
\begin{obs}\label{obs:ranges}
    For every integer $k$ with $k \geq 2$, for each superadditive and increasing function $\varphi: \N \rightarrow \N$, and for each $a \in \N$, we have
        \[\range_{\calS_\varphi}(a) = a + \varphi(a), \ \ \range_{k\calS_1}(a) = 2^ka, \ \ \range_{\calS_2}(a) = 2^aa.\]
\end{obs}
We do not attach the proof of this observation as it is straightforward, however, we encourage the reader to verify the above values for a better understanding of the structure of the families. 
Sometimes, we will use this observation implicitly.
The formula for $\range_{\calS_3}$ is not as clean, although, using a simple inequality $2^a \leq \range_{\calS_2}(a) \leq 2^{2a}$, we obtain a useful estimation.
\begin{obs}\label{obs:range:S_3}
    For every $a \in \N$, we have
        \[\tau(a,a) \leq \range_{\calS_3}(a) \leq \tau(a,3a).\]
\end{obs}
In our consideration, we will be particularly interested in the sets that are maximal in a given regular family $\calF$. 
We call such sets \emph{$\calF$-full}.
The main reason why such sets are interesting is the fact that they have to be sufficiently large.
For our arguments in the next sections, we also need some more technical notions concerning full sets.
For each regular family $\calF$, for each family $\calG$ of subsets of $\N$, and for all $a,b \in \N$ we define:
\begin{align*}
    \full(\calF) &:= \{F \in \calF : \ \  n \in \N \backslash F \Longrightarrow F \cup \{n\} \notin \calF\}, \\
    [a,b]\calG &:= \{F \in \calG : \ \ F \subset [a,b], \ \ a \in F\},\\
    \full_{a,b}(\calF) &:= \{F \in [a,b]\calF : \ \ n \in [a,b]\backslash F \Longrightarrow F \cup \{n\} \notin \calF\}.
\end{align*}
In the special case of $\calF = \calS_1$, we write that $F \subset \N$ is a \emph{full Schreier set} if $F \in \full(\calS_1)$. As mentioned, the main feature of full sets is the fact that are reasonably large. One can verify the following two observations.
\begin{obs}\label{lem:full-sets-sequence-in-S_2}
    Let $a,b,s \in \N$, and let $F_1,\dots,F_s$ be full Schreier sets such that $F_1 < \dots < F_s$.
    If $F_1\cup \dots \cup F_s \subset [a,b]$, then $b \geq 2^s a$.
\end{obs}

\begin{obs}\label{lem:full-sets-sequence-in-S_3}
    Let $a,b,s \in \N$, and let $F_1,\dots,F_s \in \full(\calS_2)$ be such that $F_1 < \dots < F_s$.
    If $F_1\cup \dots \cup F_s \subset [a,b]$, then $b \geq \tau(s,a)a$.
\end{obs}

\cref{lem:full-sets-sequence-in-S_3} implies that disjoint $\calS_2$-full sets need a lot of space. 
Now, we want to argue that having reasonably large space, we can fit many disjoint $\calS_2$-full sets. 
Note that in the case of full Schreier sets an analogous computation is straightforward. 
The below requires some technical computation.

\begin{lemma}\label{obs:lower:S3}
    Let $a \in \N$. For every $s \in \N$, there exist $F_1,\dots,F_s \in \full(\calS_2)$ with $F_1 < \dots < F_s$ such that $F_i \subset [a,\tau(s,2a+s-1)-1]$ for each $i \in [s]$. 
\end{lemma}
\begin{proof}
    First, we claim that for every $m \in \N$, the interval $[m,\tau(1,2m)-1]$ contains an $\calS_2$-full set. 
    Indeed, $[m,2^mm-1] \in \full(\calS_2)$ and $2^mm \leq 2^{2m} = \tau(1,2m)$.
    We proceed by induction on $s$.
    If $s=1$, then we use the above claim directly for $m=a$.
    Assume that $s>1$ and that the assertion holds for $s-1$, namely, the interval $[a,\tau(s-1,2a+s-2)-1]$ contains some $F_1,\dots,F_{s-1} \in \full(\calS_2)$ with $F_1 < \dots < F_{s-1}$.
    By the initial claim applied to $m = \tau(s-1,2a+s-2)$, we obtain that $[\tau(s-1,2a+s-2), \tau(1,2\tau(s-1,2a+s-2))-1]$ contains an $\calS_2$-full set. Note that
        \[\tau(1,2\tau(s-1,2a+s-2)) \leq \tau(1,\tau(s-1,2a+s-1)) = \tau(s,2a+s-1).\]
    By taking the $\calS_2$-full set in the interval as $F_s$, we finish the proof.
\end{proof}

Let us comment a little on differences between $\full(\calF)$ and $\full_{a,b}(\calF)$. 
By definition, $[a,b]\full(\calF) \subset \full_{a,b}(\calF)$. 
In general, the inclusion can be strict. 
The simplest way to see this is to take $a,b$ with $|b-a|$ small, and the set $[a,b]$. 
For instance,
\[\{7,8,9\} \in \full_{7,9}(\calS_1) \ \ \ \mrm{and} \ \ \ \{7,8,9\} \notin \full(\calS_1).\]
For the Schreier family, one can prove that all sets in $\full_{a,b}(\calF) \backslash [a,b]\full(\calF)$ are of this type (see the lemma below), however, this is not always the case. For instance,
\[\{2,3,5,6,7,8\} \in \full_{2,8}(2\calS_1) \ \ \ \mrm{and} \ \ \ \{2,3,5,6,7,8\} \notin \full(2\calS_1).\]

\begin{lemma}\label{lem:full:Sf}
    Let $\varphi: \N \rightarrow \N$ be an increasing and superadditive function. Let $a,b \in \N$ be such that $[a,b] \notin \Sf$ and let $F \subset \N$. If $F \in \full_{a,b}(\Sf)$, then $F \in [a,b]\full(\Sf)$. In particular, $\full_{a,b}(\Sf) = [a,b]\full(\Sf)$.
\end{lemma}
\begin{proof}
    Let $F \in \full_{a,b}(\Sf)$, and suppose that $F \notin \full(\Sf)$, that is, $|F| < \varphi(\min F) = \varphi(a)$ Since $[a,b] \notin \Sf$, there exists $m \in [a,b] \backslash F$. It follows that $F \cup \{m\} \in [a,b]\Sf$, which is a contradiction.
\end{proof}

The sets in $k \calS_1$ and $\calS_2$ can be seen as unions of some number of Schreier sets.
In general the ingredients of the union are not uniquely defined, however, one can fix them to be unique by a simple greedy process described below.
For every $F \subset \N$, and for every $i \in \N$ we define $E_i(F)$ with the following inductive procedure. Let $F \subset \N$. First, if $F = \emptyset$, then we set $E_1(F) := \emptyset$. Otherwise, we set $E_1(F)$ to be $F$ if $|F| \leq \min F$, and to be the first $\min F$ elements of $F$ if $|F| > \min F$. Now, let $i \in \N$, and assume that $E_1(F),\dots,E_{i}(F)$ are already defined. We set $E_{i+1}(F) := E_1(F \backslash E_{i}(F))$. For instance, for $F = [10]$, we have
    \begin{align*}
        E_1(F) = \{1\}, \ E_2(F) = \{2,3\}, \ E_3(F) = \{4,5,6,7\}, \ &E_4(F) = \{8,9,10\},\\
        &\mrm{and} \ E_5(F) = E_6(F) = \dots = \emptyset.
    \end{align*}
Let $F \subset \N$. Note that if $E_i(F) = \emptyset$ for some $i \in \N$, then $E_{i+1}(F) = E_{i+2}(F) = \dots = \emptyset$. 
Using the operators $E_i$ one can characterize sets in the families $k \calS_1$ and $\calS_2$. 
Observe that $F \in \calS_1$ if and only if $E_2(F) = \emptyset$, next, for every $k \in \N$, we have $F \in k\calS_1$ if and only if $E_{k+1}(F) = \emptyset$, and finally, $F \in \calS_2$ if and only if $E_{\min F + 1}(F) = \emptyset$.

Now, analogously to \cref{lem:full:Sf}, we study the sets in $\full_{a,b}(\calF)$ assuming that $[a,b] \notin \calF$, where $\calF$ is either $k\calS_1$ or $\calS_2$. 
Intuitively, we prove that such full sets are large.

\begin{lemma}\label{lem:full-sets-in-kS_1}
    Let $k$ be a positive integer with $k \geq 2$, let $a,b \in \N$ be such that $[a,b] \notin k\calS_1$, and let $F \subset \N$.
    If $F \in \full_{a,b}(k\calS_1)$, then $E_1(F),\dots,E_{k-1}(F)$ are full Schreier sets.
\end{lemma}
\begin{proof}
    Let $i$ be the least positive integer such that $E_i(F)$ is not a full Schreier set.
    If there exists $m \in [\min E_i(F), b] \backslash F$, then $F \cup \{m\} \in [a,b](k\calS_1)$, which is a contradiction, hence, $[\min E_i(F), b] \subset F$. 
    It follows that $E_{i+1}(F) = \emptyset$.
    Suppose that $i < k$.
    We have $[a,b] \notin k\calS_1$, thus, there exists $m \in [a,b] \backslash F$.
    Observe that $F \cup \{m\} \in [a,b](k\calS_1)$, which is again a contradiction.
    Therefore, $i = k$, which ends the proof.
\end{proof}

By repeating exactly the same proof, we obtain a similar result for $\calS_2$.

\begin{lemma}\label{lem:full-sets-in-S_2}
    Let $a,b \in \N$ be such that $[a,b] \notin \calS_2$ and let $F \subset \N$.
    If $F \in \full_{a,b}(\calS_2)$, then $E_1(F),\dots,E_{a-1}(F)$ are full Schreier sets.
\end{lemma}
 \begin{proof}
     Let $i$ be the least positive integer such that $E_i(F)$ is not a full Schreier set.
     If there exists $m \in [\min E_i(F), b] \backslash F$, then $F \cup \{m\} \in [a,b]\calS_2$, which is a contradiction, hence, $[\min E_i(F), b] \subset F$. 
     It follows that $E_{i+1}(F) = \emptyset$.
     Suppose that $i < a$.
     We have $[a,b] \notin \calS_2$, thus, there exists $m \in [a,b] \backslash F$.
     Observe that $F \cup \{m\} \in [a,b]\calS_2$, which is again a contradiction.
     Therefore, $i = a$, which ends the proof.
 \end{proof}

Intuitively, the last set (that is, $E_k(F)$, or $E_a(F)$) usually is also quite large. As we do not care much for the constants in this paper, we do not investigate this in detail in general. However, the investigation is necessary for later applications in the case of $2\calS_1$.

\begin{lemma}\label{lem:full-sets-in-kS}
    Let $a,b \in \N$ be such that $[a,b] \notin 2\calS_1$, and let $F \subset \N$.
    If $F \in \full_{a,b}(2\calS_1)$, then $E_1(F)$ is a full Schreier set and
        \[ \min E_2(F) \leq \frac{b}{2} + 2.\]
\end{lemma}
\begin{proof}
    The first part follows from \cref{lem:full-sets-in-kS_1}.
    Suppose that the second part does not hold, that is 
        \[\min E_2(F) \geq \frac{b}{2} + 3.\]
    Since $F \in \full_{a,b}(2\calS_1)$, we have $E_2(F) = [\min E_2,b]$.
    Note that by rearranging the above, we have
        \[ \min E_2(F)-2 \geq b - \min E_2(F) + 3.\]
    This yields $[\min E_2(F) - 2, b] \in \calS_1$.  
    We claim that there exists $F' \in 2\calS_1$ such that $F \subsetneq F' \subset [a,b]$.
    If $\min E_2(F) - 1 \notin E_1(F)$, then $F' := E_1(F) \cup \{\min E_2(F)-1\} \cup E_2(F)$ is a proper choice. Hence, we assume that $\min E_2(F) - 1 \in E_1(F)$.

    Suppose that $E_1(F)$ is an interval. Then, we set $F' := F_1'\cup F_2'$, where $F_1' := [\min E_1(F)-1, \min E_2(F)-3]$, and $F_2' := [\min E_2(F)-2,b]$ -- note that $\min E_1(F) - 1 \in [a,b]$ because $[a,b] \notin 2\calS_1$. Finally, we assume that $E_1(F)$ is not an interval. Let $m \in [\min E_1(F),\max E_1(F)] \backslash E_1(F)$. We set $F' := F_1' \cup F_2'$, where $F_1' := E_1(F)\cup \{m\} \backslash \{\min E_2(F)-1\}$, and $F_2' := [\min E_2(F)-1,b]$.

    This proves the claim, namely, there exists $F' \in 2\calS_1$ such that $F \subsetneq F' \subset [a,b]$, which contradicts $F \in \full_{a,b}(2\calS_1)$.
\end{proof}

The operators $E_i$ are useful to describe sets in the families $k\calS_1$ and $\calS_2$. In order to describe sets in the family $\calS_3$, we need analogous operators extracting subsequent $\calS_2$-full subsets. 
For every $F \subset \N$, and for every $i \in \N$ we define $E_i^*(F)$ with the following inductive procedure. Let $F \subset \N$. First, if $F = \emptyset$, then we set $E_1^*(F) := \emptyset$. 
Otherwise, we set $E_1^*(F)$ to be $F$ if $F \in \calS_2$, and to be $E_1(F) \cup \dots \cup E_{\min F}(F)$ if $F \notin \calS_2$. Now, let $i \in \N$, and assume that $E_1^*(F),\dots,E_{i}^*(F)$ are already defined. We set $E_{i+1}^*(F) := E_1^*(F \backslash E_{i}^*(F))$.

\begin{lemma}\label{lem:full-sets-in-S_3}
    Let $a,b \in \N$ be such that $[a,b] \notin \calS_3$, and let $F \subset \N$.
    If $F \in \full_{a,b}(\calS_3)$, then $E_1^*(F),\dots,E_{a-1}^*(F)$ are $\calS_2$-full sets.
\end{lemma}
\begin{proof}
    Let $i$ be the least positive integer such that $E_i^*(F)$ is not an $\calS_2$-full set.
    Let $a' = \min E_i^*(F)$. It is clear that $E_i^*(F) \in \full_{a',b}(\calS_2)$.
    Suppose that $i < a$.
    If $[a',b] \in \calS_2$, then there exists $m \in [a,b] \backslash F$, and $F \cup \{m\} \in \calS_3$, which is a contradiction.
    We can assume that $[a',b] \notin \calS_2$.
    By \cref{lem:full-sets-in-S_2}, $E_1(E_i^*(F)), \dots, E_{a'-1}(E_i^*(F))$ are full Schreier sets. 
    If there exists $m \in [\min E_{a'}(E_i^*(F)), b] \backslash F$, then $F \cup \{m\} \in [a,b]\calS_3$, which is a contradiction. It follows that 
        \[[\min E_{a'}(E_i^*(F)), b] \subset F,\] 
    which yields $E_{i+1}^*(F) = \emptyset$.
    Since $[a,b] \notin \calS_3$, there exists $m \in [a,b] \backslash F$.
    Observe that $F \cup \{m\} \in \calS_3$, which is again a contradiction.
    Therefore, $i = a$, which ends the proof.
\end{proof}

	\section{Tools for lower bounds}
    \label{sec:tools_lower}
    The idea for proving the lower bounds is the same for all the regular families that we consider. 
For this reason, we are going to prove an abstract lemma, and then apply it to various families. 
The general plan of constructing an element of $c_{00}$ with high $j_\calF(x)$ is to put a very high value on the first coordinate and a bunch of last coordinates. 
This way, we force every functional attaining the norm to be a sum of many functionals from $W_0(\calF)$. 
Intuitively, this leaves the largest possible space to proceed with inductive construction.
See an example in \cref{fig:lower}. 

\begin{figure} 
  \begin{center}
    \includegraphics{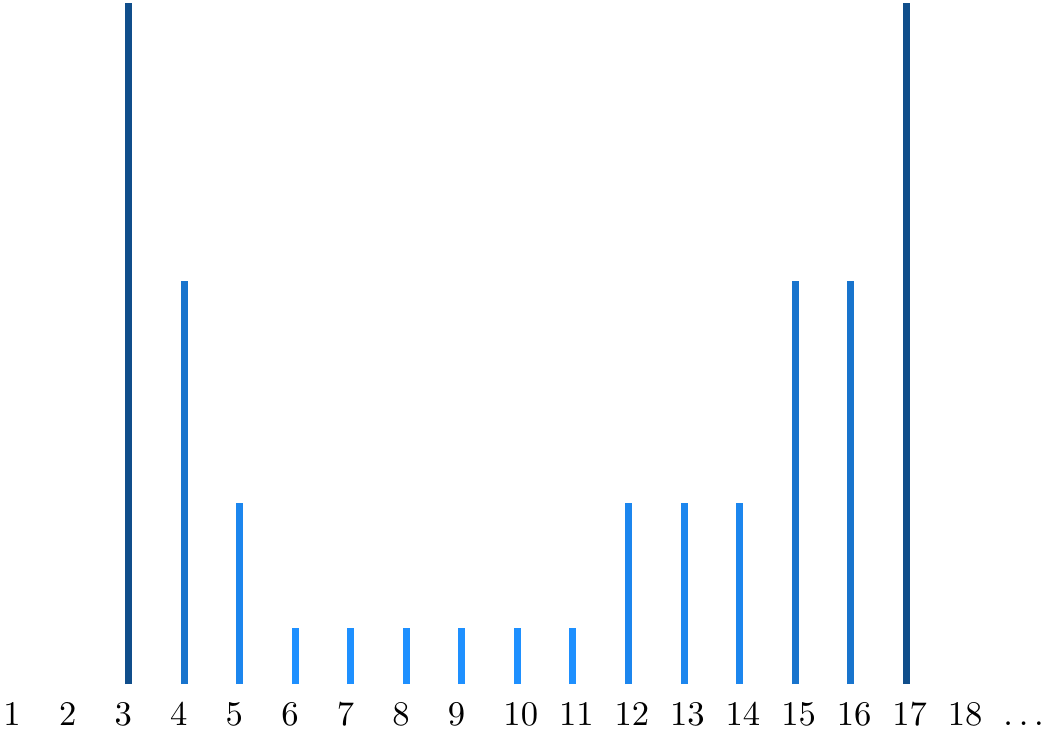} 
  \end{center}
  \caption{
    Consider the case, where $\calF = \calS_1$. Let $x$ be a sequence constructed as in the figure, that is, we put some value on coordinates $6,7,8,9,10,11$, next we put much greater value on coordinates $5,12,13,14$, and so on. The values on coordinates $3$ and $17$ are so big that we have to take them with the smallest possible weight, that is $\frac{1}{2}$, however, as the minimum coordinate is $3$ and we work with $\calS_1$, there is only room for one more functional in the sum. Hence, if $f(x) = \|x\|_{\calS_1}$, than $f = \frac{1}{2}(e_3^* + g + e_{17}^*)$ with $\supp g \subset [4,16]$. Now, we repeat the reasoning for $g$, that is, values on coordinates $4$, $15$, and $16$ are so big that we have to take them with the smallest possible weight obtaining that $g = \frac{1}{2}(e_4^* + h + e_{15}^* + e_{16}^*)$ with $\supp h \subset [5,14]$. We continue, finally obtaining $j_{\calS_1}(x) \geq 4$.
  }
  \label{fig:lower}
\end{figure}

\vbox{
\begin{lemma} \label{lem:main:lower}
Let $\calF$ be a regular family and let $d \in \N$. If there exist $F_1,\dots,F_d \in \full(\calF)$ and $a_1,b_1,\dots,a_{d-1},b_{d-1} \in \N$ such that
\begin{enumerateNumL}
    \item for all $i \in [d]$ we have $|F_i| \geq 3$,\label{item:lower:first}\label{item:lower:at-least-3}
    \item for all $i \in [d-1]$ elements $a_i,b_i$ are consecutive in $F_i$,\label{item:lower:consecutive}
    \item for all $i \in [d-1]$ we have $a_i \in F_{i+1}$ and $F_{i+1} \subset [a_i,b_i-1]$,\label{item:lower:nested}
    \item for all $i \in [d-1]$ and distinct $a,a' \in [a_i,b_i-1]$ we have $(F_i\backslash \{a_i\}) \cup \{a,a'\} \notin \calF$,\label{item:lower:last}\label{item:lower:full}
\end{enumerateNumL}
then there exists $x \in c_{00}^+$ with $\supp x \subset \bigcup_{i=1}^d F_i$ and $j_\calF(x) \geq d$, in particular \mbox{$j_\calF(\max F_1) \geq d$}.
\end{lemma}
}
\begin{proof}
    We proceed by induction on $d$. Let us start with the case of $d = 1$. Suppose that there exists $F_1 \in \full(\calF)$ satisfying items \ref{item:lower:first}-\ref{item:lower:last}, namely, we have $|F_1| \geq 3$. For each $i \in \N$ define
       \[x_i := 
        \begin{cases}
         1 & \text{if } i \in F_1, \\
         0 & \text{otherwise.}
        \end{cases} \]
    Let $x := (x_i)_{i\in \N}$. For every $e \in W_0$, we have $e(x) \leq 1$. Let $c_1,c_2,c_3 \in F$ be three distinct elements. Define $f := \frac{1}{2} (e_{c_1}^* + e_{c_2}^* + e_{c_3}^*)$. Since $\calF$ is hereditary, we have $f \in W(\calF)$. Clearly, $f(x) = \frac{3}{2}$ and $\depth_\calF(f) = 1$, hence, $j_\calF(x) \geq 1$.

    Now, let $d > 1$ and suppose that there exist $F_1,\dots,F_d \in \full(\calF)$ and $a_1,b_1,\dots,a_{d-1},b_{d-1} \in \N$ satisfying items \ref{item:lower:first}-\ref{item:lower:last}. By the inductive assumption applied to $F_2,\dots,F_d$ and $a_2,b_2,\dots,a_{d-1},b_{d-1}$, we obtain $x' \in c_{00}^+$ with $\supp x' \subset \bigcup_{i=2}^d F_i$ and $j_\calF(x') \geq d-1$. Let $s$ be the sum of all the coefficients of $x'$, and let $F' := \bigcup_{i=2}^d F_i$. By \ref{item:lower:consecutive} and \ref{item:lower:nested}, $F_1 \backslash F' = F_1 \backslash \{a_1\}$. For each $i \in \N$ we define
       \[x_i := 
        \begin{cases}
         x'_i & \text{if } i \in F', \\
         2s & \text{if } i \in F_1 \backslash F', \\
         0 & \text{otherwise.}
        \end{cases} \]
    Let $x := (x_i)_{i\in \N}$. Let $f \in W(\calF)$ be such that $f(x) = \|x\|_\calF$ and let $f' \in W(\calF)$ be such that $f'(x') = \|x'\|_\calF$. 
    Since $j_\calF(x') \geq d-1$, we have $\depth_\calF(f') \geq d-1$. We define $g := \frac{1}{2}\left(f' + \sum_{i \in F_1 \backslash F'} e_i^*\right)$. Note that $F_1 \backslash F' \cup \{\min \supp f'\} \in \calF$, thus, $g \in W(\calF)$. The goal is to prove that $f$ is of a similar form as $g$. First observe that,
        \[ g(x) = \frac{1}{2}\left( f'(x) + (|F_1\backslash F'|)\cdot 2s \right) =  \frac{1}{2} \|x'\|_\calF + (|F_1| - 1) \cdot s  > 2s. \]
    By definition, for each $i \in \N$, we have $x_i \leq 2s$, thus, if $\depth_\calF(f) = 0$, then $\|x\|_\calF = f(x) \leq 2s < g(x) \leq \|x\|_\calF$, which is a contradiction. Therefore, $\depth_\calF(f) > 0$. It follows that for each $i \in \N$, $f(e_i) \leq \frac{1}{2}$. We claim that for each $i \in F_1 \backslash F'$, we have $f(e_i) = \frac{1}{2}$. For a contradiction, suppose that $f(e_{i_0}) < \frac{1}{2}$ for some $i_0 \in F_1 \backslash F$. Since $f \in W(\calF)$ the value of $f(e_i)$ has to be an inverse of a power of $2$, thus, $f(e_i) \leq \frac{1}{4}$. We have
        \begin{align*}
            \|x\|_\calF =
            f(x) &=
            \sum_{i \in \N} f(e_i) \cdot x_i =
            \sum_{i \in F_1 \cup F'} f(e_i) \cdot x_i = \sum_{i \in (F_1 \cup F') \backslash \{i_0\}} f(e_i) \cdot x_i + f(e_{i_0}) x_{i_0} \\
            &\leq \sum_{i \in (F_1 \cup F') \backslash \{i_0\}} \frac{1}{2} \cdot x_i + \frac{1}{4} x_{i_0} = 
            \sum_{i \in F_1 \backslash (F' \cup \{i_0\})} \frac{1}{2} \cdot x_i + \sum_{i \in F'} \frac{1}{2} \cdot x_i + \frac{1}{4} x_{i_0} \\
            &= (|F_1| - 2) \cdot s + \frac{1}{2} s + \frac{1}{4} \cdot 2s = s \cdot (|F_1| - 1) < g(x) \leq \|x\|_\calF.
        \end{align*}
    This is a contradiction, and so, for each $i \in F_1 \backslash F'$, we have $f(e_i) = \frac{1}{2}$. In particular,
        \[f = \frac{1}{2}\left(f_1 + \dots + f_m + \sum_{i \in F_1\backslash F'} e_i^*\right)\]
    for some $m \in \N_0$ and $f_1,\dots,f_m \in W(\calF)$ such that
        \[F_1\backslash \{a_1\} \cup \{\min \supp f_1 ,\dots, \min \supp f_m\} \in \calF.\]
    However, by \ref{item:lower:full}, the above gives 
        \[|\{\min \supp f_1 ,\dots, \min \supp f_m\} \cap [a_1,b_1-1]| \leq 1.\] 
    By comparing $f(x)$ with $g(x)$, we have $f_1(x') + \dots f_m(x') \geq \|x'\|_\calF > 0$.
    Clearly, if for some $\ell \in [m]$, $\min \supp f_\ell \notin [a_1,b_1-1]$, then $f_\ell(x') = 0$, and so, there exists $\ell \in [m]$ such that $\min \supp f_\ell \in [a_1,b_1-1]$. Since $f_\ell(x') = \|x'\|_\calF$, we have $\depth_\calF(f_\ell) \geq d - 1$, and so, $\depth_\calF(f) \geq d$, which ends the proof.
\end{proof}

As explained in the caption of \cref{fig:lower}, the strategy is to take $a_1 = 3, a_2 = 4, a_3 = 5$, and so on.
We are almost ready to proceed with the construction of the sequences of sets for some regular families.
The last remaining detail to take care of is to make sure that the families that we consider satisfy \ref{item:lower:full}.
To this end, we abstract the following property of a regular family.
We say that a regular family $\calF$ is \emph{strong} if for every integer $a$ with $a \geq 3$, and for all integers $b,c$ with $a + 1 < b \leq c$ such that $[a,a+1]\cup [b,c] \in \full(\calF)$, for all distinct $a',a'' \in [a+2,b-1]$ we have $\{a,a',a''\} \cup [b,c] \notin \calF$. The following is immediate to check.
\begin{obs}\label{obs:strong}
    For every increasing and superadditive function $\varphi: \N \rightarrow \N$ and for every integer $k$ with $k \geq 2$, the families $\Sf, k \calS_1,\calS_2,\calS_3$ are strong.
\end{obs}
As the construction of the sequence $F_1,\dots,F_d$ is virtually the same for all the families that we consider, we For all $t,s \in \N$ with $s + 1 < t$, we define
    \[r_\calF(s,t) := \min \{m \in \N : \{s,s+1\} \cup [t,m-1] \in \full(\calF)\}.\]
Next, for all $s,t \in \N$ with $s + 1 < t$ and for each $u \in \N_0$ we define
    \[q_\calF(u,s,t) = \begin{cases}
         t & \text{if } u > s, \\
         r_\calF(s,t) & \text{if } u = s, \\
         r_\calF(u,q_\calF(u+1,s,t)) & \text{if } u \leq s.
        \end{cases}\]
For example, $q_\calF(3,5,10) = r_\calF(3,r_\calF(4,r_\calF(5,10)))$.
The definition of $q_\calF$ is a little convoluted, and to understand its purpose one should read the next lemma.

\begin{lemma}\label{lem:lower:greater:abstract}
    Let $\calF$ be a strong regular family such that $\range_\calF(3) \geq 3$. For every $n,d \in \N$ if $q_\calF(3,d+2,d+4) \leq n$, then
        \[d \leq j_\calF(n).\]
\end{lemma}
\begin{proof}
    Let $n,d \in \N$, be such that $q_\calF(3,d+2,d+4) \leq n$. For each $i \in [d]$, we define
        \[F_i := \{i+2,i+3\} \cup [q_\calF(i+3,d+2,d+4),q_\calF(i+2,d+2,d+4)-1].\]
    It follows that $F_i \in \full(\calF)$ and $|F_i| \geq 3$ (so \ref{item:lower:at-least-3} is satisfied).
    If $i < d$, then we define $a_i := i+2$ and $b_i := q_\calF(i+3,d+2,d+4)$.
    Observe that $a_i \in F_{i+1}$ and $F_{i+1} \subset [a_i,b_i-1]$ (so \ref{item:lower:nested} is satisfied).
    Item~\ref{item:lower:consecutive} is clearly satisfied.
    Since $\calF$ is strong, \ref{item:lower:full} is also satisfied. Therefore, by \cref{lem:main:lower}, for every $n \in \N$ such that $q_\calF(3,d+2,d+4) \leq n$, we have
        \[d \leq j_\calF(\max F_1) \leq j_\calF(q_\calF(3,d+2,d+4)) \leq j_\calF(n).\qedhere\]
\end{proof}
    
	\section{Lower bounds}
    \label{sec:lower}
    In this section we apply \cref{lem:lower:greater:abstract} to the families $\calS_1,\calS_\varphi,k\calS_1,\calS_2$. 
In each case, we establish some bounds on $q_\calF(3,d+2,d+4)$, and then compare the bound to $n$, in order to obtain the final lower bound on $j_\calF(n)$ by applying \cref{lem:lower:greater:abstract}. 


\subsection{Lower bound for $\Sf$}
Let $\varphi: \N \rightarrow \N$ be a non-decreasing. Fix some $u,s,t \in \N$ with $s + 1 < t$. It is clear that 
    \[r_{\Sf}(s,t) = t + \varphi(s) - 2.\]
It follows that $q_{\Sf}(u,s,t) \leq t + \sum_{i=u}^s (\varphi(i)-2)$. For every integer $d \in \N$, we have
    \[q_{\Sf}(3,d+2,d+4) = d+4 + \sum_{i=3}^{d+2} (\varphi(i)-2) = \sum_{i=3}^{d+2} \varphi(i) - d + 4 \leq \sum_{i=3}^{d+3} \varphi(i).\]
By \cref{lem:lower:greater:abstract}, if $\sum_{i=3}^{d+3} \varphi(i) \leq n$, then $j_{\Sf}(n) \geq d$ for every $n \in \N$. 
\begin{corollary}\label{cor:lower:Sf}
    For each $n \in \N$, we have
        \[ j_{\calS_\varphi}(n) \geq \max\{ m \in \N : \sum_{j=3}^{m+3}\varphi(j)<n \}.\]
\end{corollary}

Substituting $\varphi = \mathrm{id}$ in the above, gives a lower bound for $j_{\calS_1}(n)$.
\begin{corollary}\label{cor:classic-schreier}
    For each $n \in \N$, we have
        \[ j_{\calS_1}(n) \geq \sqrt{2n} - 3.\]
\end{corollary}

\subsection{Lower bound for $k\calS_1$}

Let $k$ be an integer with $k \geq 2$. Fix some $u,s,t \in \N$ with $s + 1 < t$. It is clear that 
    \[r_{k\calS_1}(s,t) = 2^{k-1}(t+s-2) \leq 2^{k} t.\]
It follows that $q_{k\calS_1}(u,s,t) \leq 2^{k(s-u+1)}t \leq 2^{k(s-u+1) + t}$. Let $n \in \N$. For every integer $d \in \N$, we have
    \[q_{k\calS_1}(3,d+2,d+4) \leq 2^{(k+1)d + 4}. \]
By \cref{lem:lower:greater:abstract}, if $2^{(k+1)d + 4} \leq n$, then $j_{k\calS_1}(n) \geq d$ for every $n \in \N$.
\begin{corollary}\label{cor:lower:kS1}
    For each $n \in \N$, and for each integer $k$ with $k \geq 2$, we have
        \[ j_{k\calS_1}(n) \geq \frac{1}{k+1} \log n - \frac{4}{k+1}-1.\]
\end{corollary}


\subsection{Lower bound for $\calS_2$}
Fix some $u,s,t \in \N$ with $s+1 < t$. It is clear that
    \[ r_{\calS_2}(s,t) = 2^{s-1}(t+s-2) \leq 2^{s} t.\]
It follows that $q_{\calS_2}(u,s,t) \leq \prod_{i=u}^s 2^{i} t \leq 2^{(s(s+1))\slash 2 + t}$. For every integer $d \in \N$, we have
    \[q_{\calS_2}(3,d+2,d+4) \leq 2^{(d^2+7d+14)\slash 2}. \]
By \cref{lem:lower:greater:abstract}, if $2^{(d^2+7d+14)\slash 2} \leq n$, then $j_{\calS_2}(n) \geq d$ for every $n \in \N$.
\begin{corollary}\label{cor:lower:S2}
    For each $n \in \N$, we have
        \[ j_{\calS_2}(n) \geq \sqrt{2\log n}  - 5.\]
\end{corollary}


\subsection{Lower bound for $\calS_3$}
Fix some $u,s,t \in \N$ with $s+1 < t$. To estimate $r_{\calS_3}(s,t)$, note that the full set $F := \{s,s+1\} \cup [t,r_{\calS_3}(s,t)-1]$ consists of two parts. The first part is a prefix of $F$ that is an $\calS_2$-full set. The second part is the rest of $F$, it starts after the element $t2^t$, and it is the union of $s-1$ pairwise disjoint $\calS_2$-full sets. By \cref{obs:lower:S3}, we have
    \[ r_{\calS_3}(s,t) \leq \tau(s-1,2t2^t+s-1) \leq \tau(s-1,2^{2t+2}) =\tau(s,2t+2) \leq \tau(s+1,t).\]
It follows that $q_{\calS_3}(u,s,t) \leq \tau(\sum_{i=u}^s (i+1), t) \leq \tau((s+2)^2\slash 2, t)$. For every integer $d \in \N$, we have
    \[q_{\calS_3}(3,d+2,d+4) \leq \tau((d+4)^2\slash 2, d+4). \]
By \cref{lem:lower:greater:abstract}, if $\tau((d+4)^2\slash 2, d+4) \leq n$, then $j_{\calS_3}(n) \geq d$ for every $n \in \N$.
\begin{corollary}\label{cor:lower:S3}
    For each $n \in \N$, we have
        \[ j_{\calS_3}(n) \geq \sqrt{2\log^* n}  - 5.\]
\end{corollary}

	\section{Tools for upper bounds}
    \label{sec:tools_upper}

    \subsection{Some auxiliary definitions and simple observations}
    Let $\calF$ be a regular family, let $x \in c_{00}$, and let $f,g \in W(\calF)$. We write 
    \[\spn x := [\min \supp x, \max \supp x] \ \ \ \mrm{and} \ \ \ \spn f := [\min \supp f, \max \supp f].\]
We say that $f$ is \emph{$\calF$-realizing for $x$} if 
\begin{itemize}
    \item $f(x) = \|x\|_\calF$,
    \item $\depth_\calF(f) = j_\calF(x)$, and
    \item $\spn f \subset \spn x$.
\end{itemize}
We say that \emph{$g$ is not $\calF$-worse than $f$ for $x$} if 
\begin{itemize}
    \item $g(x) \geq f(x)$,
    \item $\depth_\calF(g) \leq \depth_\calF(f)$, and
    \item $\spn g \subset \spn f$.
\end{itemize}
Observe that if $f$ is $\calF$-realizing for $x$ and $f$ is not $\calF$-worse than $g$ for $x$, then $g$ is $\calF$-realizing for $x$. Moreover, this relation is transitive. We will use these facts implicitly and repeatedly.
We define
    \[\full_f(\calF) := \full_{\min \supp f, \max \supp f}(\calF).\]
Let $m \in \N$ and let $f_1,\dots,f_m \in W(\calF)$. We say that $(f_1,\dots,f_m)$ is \emph{$\calF$-full-building for $f$} if $(f_1,\dots,f_m)$ is $\calF$-building for $f$ and 
    \[\{\min \supp f_i : i \in [m]\} \in \full_{f}(\calF).\]
We say that $f$ is \emph{$\calF$-full} if there exist a positive integer $m$ and $f_1,\dots,f_m \in W(\calF)$ such that $(f_1,\dots,f_m)$ is $\calF$-full-building for $f$.
For each $i \in \N$, we define
       \[(x|f)_i := 
        \begin{cases}
         x_i & \text{if } \min \supp f \leq i \leq \max \supp f, \\
         0 & \text{otherwise,}
        \end{cases} \]
and we let $(x|f) := ((x|f))_{i\in \N}$.

\begin{obs}\label{lem:force:F}
    Let $x \in c_{00}^+$, let $\calF$ be a regular family, and let $f \in W(\calF)$. 
    If $\supp f \in \calF$ then there exists $g \in W(
    \calF)$ with $\depth_\calF(g) \leq 1$ that is not $\calF$-worse than $f$ for $x$. 
    
    In particular, if $[a,b] \in \calF$ for some $a,b \in \N$ with $a \leq b$, then $j_\calF(a,b) \leq 1$.
\end{obs}
\begin{proof}
    Assume that $\supp f \in \calF$. 
    If $\depth_\calF(f) \leq 1$, then $g:=f$ satisfies the assertion.
    Otherwise, we put $g := \frac{1}{2}\sum_{i \in \supp f} e_i^*$. 
    Clearly, $g(x) \geq f(x)$ and $\depth_\calF(g) = 1$.
\end{proof}

Next, we prove that for each $x \in c_{00}^+$, the norm $\|x\|_\calF$ is always realized wither by a very shallow functional, of by an $\calF$-full functional.

\begin{lemma}\label{lem:force:star}
    Let $x \in c_{00}^+$ and let $\calF$ be a regular family. For every $f \in W(\calF)$ such that $f(x) = \|x\|_\calF$ there exists $g \in W(\calF)$ that is not $\calF$-worse than $f$ for $x$, and either
    \begin{enumerateNumF}
        \item $\depth_\calF(g) \leq 1$ or\label{lem:force:star:depth}
        \item $g$ is $\calF$-full.\label{lem:force:star:full}
    \end{enumerateNumF}
\end{lemma}
\begin{proof}
    Suppose that the assertion does not hold.
    Let us choose a counterexample $f \in W(\calF)$ satisfying the premise of the lemma according to the following rule.
    For each $f$ being a counterexample choose the maximum $d \in \N$ such that there exist $f_1,\dots,f_d$ with $(f_1,\dots,f_d)$ being $\calF$-building for $f$.
    Now, we choose $f$ such that the value $d$ is maximum among all counterexamples $f$. 
    Fix $f_1,\dots,f_d \in W(\calF)$ as above, and let $F_f := \{\min \supp f_i : i \in [d]\}$.
    Observe that, as $f$ can not be taken to be $g$, therefore, $\depth_\calF(f) \geq 2$ and $F_f \notin \full_{f}(\calF)$.
    It follows that there exists $n \in \spn f\backslash F_f$ with $F_f \cup \{n\} \in \calF$. Let $n^*$ be the maximum such number.  
    Let $t$ be the maximum number in $[d]$ such that $\max\supp f_t \leq n^*$. 
    First, suppose that $\max\supp f_t < n^*$, then we define
        \[ f':= \frac{1}{2}\left( \sum_{i=1}^t f_i + e_{n^*}^* + \sum_{i=t+1}^d f_i \right).\]
    Next, suppose that $\max\supp f_t = n^*$. Note that $f_t \neq e_{n^*}$, and so $f_t' := f_t|_{[1,n^*-1]}$ is a well-defined member of $W(\calF)$. We define
        \[ f':= \frac{1}{2}\left( \sum_{i=1}^{t-1} f_i + f_t' + e_{n^*}^* + \sum_{i=t+1}^d f_i \right).\]
    Since $F \cup \{n^*\} \in \calF$, in both cases $f' \in W(\calF)$. 
    Moreover, $f(x) \leq f'(x)$ in both cases, and in particular, $f'(x) = \|x\|_\calF$. Furthermore, $\depth_\calF(f') = \depth_\calF(f)$ and $\spn f' = \spn f $. 
    It follows that $f'$ is not $\calF$-worse than $f$ for $x$.
    By the choice of $f$, the functional $f'$ is not a counterexample, and so, there exists $g \in W(\calF)$ not $\calF$-worse than $f'$ for $x$ that satisfies~\ref{lem:force:star:depth}~or~\ref{lem:force:star:full}.
    However, we obtain that $g$ is not $\calF$-worse than $f$ for $x$, which contradicts the fact that $f$ is a counterexample.
\end{proof}
    
	\subsection{The insertion property}
    In this section, the main goal is to generalize the core step of the proof of an upper bound for $j_{\calS_1}(n)$ by Beanland, Duncan, and Holt \mbox{\cite[Lemma~1.8]{Beanland_2018}}. 
We reprove the result with $\calS_1$ replaced with a regular family satisfying a certain abstract property. 
More precisely, we aim to develop a property of a regular family such that assuming it, we can strengthen condition~\ref{lem:force:star:full} in \cref{lem:force:star}.

We say that a regular family $\calF$ \emph{has the insertion property} if for all $a,s,t \in \N$ with $s,t \geq 2$, and for all $n_2,\dots,n_s,m_2,\dots,m_t \in \N$ with $m_2 < \dots < m_t < n_2 < \dots < n_s$ if 
    \[\{a,m_2,\dots,m_t\} , \{a,n_2,\dots,n_s\}\in \calF \ \ \ \mathrm{and} \ \ \ \range_\calF(a) < m_2,\]
then
    \[\{m_2,m_3,\dots,m_t,n_2,\dots,n_s \} \in \calF.\]
First, we show that the families $\Sf,\calS_2$, and $\calS_3$ have the insertion property. Here, we use the superadditivity of $\varphi$.

\begin{lemma}\label{lem:S-y-nice}
    Let $\varphi: \N \rightarrow \N$ be increasing and superadditive. The family $\calS_\varphi$ has the insertion property.
\end{lemma}
\begin{proof}
    Let $a,s,t \in \N$ be such that $s,t \geq 2$, and let and $n_2,\dots,n_s,m_2,\dots,m_t \in \N$ with $m_2 < \dots < m_t < n_2 < \dots < n_s$.
    Let $F := \{a,n_2,\dots,n_s\}$ and $ G := \{a,m_2,\dots,m_t\}$.
    Assume that $F,G \in \Sf$ and $\range_{\Sf}(a) < m_2$. Let $H := \{m_2,m_3,\dots,m_t,n_2,\dots,n_s \}$.
    Since $\range_{\Sf}(a) = \varphi(a) + a < m_2$, we have $\varphi(\varphi(a) + a) < \varphi(m_2)$, moreover, $\varphi(a) + \varphi(a) \leq \varphi(\varphi(a) + a)$. 
    Since $F,G \in \calS_\varphi$, we have $t = |F| \leq \varphi(\min F) = \varphi(a) $ and $s = |G| \leq \varphi(\min G) = \varphi(a) $. Therefore,
        \[|H| = s+t-2 < s + t \leq \varphi(a) + \varphi(a) \leq \varphi(\varphi(a) + a) < \varphi(m_2) = \varphi(\min H).\]
    This yields $H \in \calS_\varphi$.
\end{proof}

\begin{lemma}\label{lem:S-3-nice}\label{lem:S-2-nice}
    For each $\ell \in \{2,3\}$, the family $\calS_\ell$ has the insertion property. 
\end{lemma}
\begin{proof}
    Let $a,s,t \in \N$ be such that $s,t \geq 2$, and let and $n_2,\dots,n_s,m_2,\dots,m_t \in \N$ with $m_2 < \dots < m_t < n_2 < \dots < n_s$.
    Let $F := \{a,n_2,\dots,n_s\}$ and $ G := \{a,m_2,\dots,m_t\}$.
    Assume that $F,G \in \Sf$ and $\range_{\calS_\ell}(a) \leq m_2$. Let $H := \{m_2,m_3,\dots,m_t,n_2,\dots,n_s \}$.
    There exist $F_1,\dots,F_a,G_1,\dots,G_a \in \calS_{\ell-1}$ such that $F_1 < \dots < F_{a}$, $G_1 < \dots < G_{a}$, and $F = F_1 \cup \dots \cup F_a, G = G_1 \cup \dots \cup G_a$. Observe that
        \[H = G_1 \backslash \{a\} \cup G_2 \cup \dots \cup G_a \cup F_1 \backslash \{a\} \cup F_2 \cup \dots \cup F_a.\]
    Since $\range_{\calS_\ell}(a) \leq m_2 = \min H$, in order to prove that $H \in \calS_\ell$, it suffices to check if $2a < \range_{\calS_\ell}(a)$, which is clear in both cases by \cref{obs:ranges}. (Note that $F \in \calS_\ell$ requires $a > 1$.)
\end{proof}

Observe that the family $2 \calS_1$ does not have the insertion property. Indeed, consider the following example:
    \begin{align*}
        F := \{2,99\}\cup [100,199] \ \ \ \mathrm{and} \ \ \ G := \{2,19\} \cup [20,39].
    \end{align*}
Clearly, $F,G \in 2\calS_1$. The greatest element of $G$, that is, $39$ is less than the second least element of $F$, that is $99$. It is easy to compute that $\range_{2\calS_1}(2) = 6$, thus the inequality $\range_{\calS_2}(a) < m_2 - a + 1$ takes form $6 < 19 - 2 + 1$, which is clearly true. The insertion property would give
    \[\{19\} \cup [20,39] \cup \{99\} \cup [100,199] \in 2\calS_1.\]
This can be easily verified to be false. Following a similar idea, one can construct counterexamples showing that $k\calS_1$ does not have the insertion property for all $k \geq 2$. This fact indicates that the family $k\calS_1$ has to be treated differently.

As already announced, we now strengthen \cref{lem:force:star} for regular families that have the insertion property.

\begin{lemma}\label{lem:upper:main}
    Let $x \in c_{00}^+$ and let $\calF$ be a regular family that has the insertion property. For every $f \in W(\calF)$ such that $f(x) = \|x\|_\calF$ there exists $g \in W(\calF)$ that is not $\calF$-worse than $f$ for $x$ and either
    \begin{enumerateNumSF}
        \item $\depth_\calF(g) \leq 1$ or \label{lem:upper:main:i}
        \item $\depth_\calF(g) \geq 2$ and there exists a positive integer $d$ and $g_1,\dots,g_d \in W(\calF)$ such that $(g_1,\dots,g_d)$ is $\calF$-full-building for $g$, and either $\depth_\calF(g_1) = 0$ or there exist a positive integer $e$ and $t_1,\dots,t_e \in W(\calF)$ such that $(t_1,\dots,t_e)$ is $\calF$-building for $g_1$, and $\depth_\calF(t_1) \leq 1$.\label{lem:upper:main:iii}
    \end{enumerateNumSF}
\end{lemma}
\begin{proof}
    By \cref{lem:force:star}, there exists $g' \in W(\calF)$ that is not $\calF$-worse than $f$ for $x$, and either $\depth_\calF(g') \leq 1$, or $g'$ is $\calF$-full. Fix such $g'$ with $\min \supp g'$ maximal. 
    If $\depth_\calF(g') \leq 1$, then \ref{lem:upper:main:i} is satisfied for $g:=g'$. Therefore, we can assume that $\depth_\calF(g') \geq 2$, and that $g'$ is $\calF$-full for $x$. 
    Let $d \in \N$ and let $g_1,\dots,g_d \in W(\calF)$ be such that $(g_1,\dots,g_d)$ is $\calF$-full-building for $g'$. 
    If $\depth_\calF(g_1) = 0$, then \ref{lem:upper:main:iii} is satisfied. Thus, we can assume that $\depth_\calF(g_1) \geq 1$. Let $e \in \N$ and let $t_1,\dots,t_e \in W(\calF)$ be such that $(t_1,\dots,t_e)$ is $\calF$-building for $g_1$. If $\depth_\calF(t_1) \leq 1$, then \ref{lem:upper:main:iii} is satisfied, hence, we assume that $\depth_\calF(t_1) \geq 2$. Let $a := \min \supp t_1$.

    If $\max \supp t_1 < \range_\calF(a)$, then $\supp t_1 \in \calF$, and so, by \cref{lem:force:F}, there exists $t_1'$ that is not $\calF$-worse than $t_1$ for $x$. Let $g$ be obtained from $g'$ by replacing $t_1$ with $t_1'$. Item~\ref{lem:upper:main:iii} is satisfied, thus, we assume that $\range_\calF(a) \leq \max \supp t_1$, and so, $\range_\calF(a) < \min \supp t_2$.
    Since $\calF$ has the insertion property, we have
        \[H := \{\min\supp t_2, \dots, \min\supp t_e, \min\supp g_2, \dots, \min\supp g_d\} \in \calF.\]
    This yields
    \begin{align*}
        h_1 &:= \frac{1}{2} (t_2 + \dots + t_e + g_2 + \dots + g_d) \in W(\calF) \ \ \mrm{and}\\
        h_2 &:= \frac{1}{2} (t_1 + g_2 + \dots + g_d) \in W(\calF).
    \end{align*}
    We have $\frac{1}{2}(h_1 + h_2) = g'$ and $\|x\|_\calF = f(x) \leq g'(x) \leq \|x\|_\calF$. Therefore, $\|x\|_\calF = g'(x) =h_1(x) = h_2(x)$. It is easy to verify that $h_1$ is not $\calF$-worse than $f$ for $x$. However, by \cref{lem:force:star}, this yields the existence of $h_1' \in W(\calF)$ such that $h_1'$ is not $\calF$-worse than $h_1$ for $x$ and either $\depth_\calF(h_1') \leq 1$, or $h_1'$ is $\calF$-full. Note that $\min \supp g' < \min \supp h_1 \leq \min \supp h_1'$, which contradicts the choice of $g'$.
\end{proof}
    
	\subsection{Optimal sequences of realizing functionals}
    In the final part of this section, we inductively apply \cref{lem:force:star} and \cref{lem:upper:main} in order to derive ``optimal sequences'' of realizing functionals for each $x \in c_{00}^+$. First, we need the following technical observation.

\begin{obs}\label{obs:realizing_2_subfunctional}
    Let $x \in c_{00}^+$ and let $\calF$ be a regular family. Let $g \in W(\calF)$ be $\calF$-realizing for $x$ with $\depth_\calF(g) \geq 4$. Let $d \in \N$ and let $g_1,\dots,g_d \in W(\calF)$ be such that $(g_1,\dots,g_d)$ is $\calF$-building for $g$. For each $i \in [d]$ with $\depth_\calF(g_i) = 0$, let $d_i := 0$; for each $i \in [d]$ with $\depth_\calF(g_i) \geq 1$ let $d_i \in \N$ be such that there exist $t_1^{(i)},\dots,t_{d_i}^{(i)} \in W(\calF)$ with $(t_1^{(i)},\dots,t_{d_i}^{(i)})$ being $\calF$-building for $g_i$. Then, there exists $i \in [d]$ and $j \in [d_i]$ such that
    \begin{enumerateNumO}
        \item if $\depth(g_1) = 0$ or $\depth_\calF(t_1^{(1)}) \leq 1$, then $(i,j) \neq (1,1)$;\label{obs:realizing_2_subfunctional:item:not_1}\label{obs:realizing_2_subfunctional:item:first}
        \item $\depth_\calF(g) = \depth_\calF(t_j^{(i)}) + 2$;\label{obs:realizing_2_subfunctional:item:depth}
        \item $t_j^{(i)}$ is $\calF$-realizing for $x|t_j^{(i)}$;\label{obs:realizing_2_subfunctional:item:realizing}
        \item $|\spn t_j^{(i)}| < |\spn x|$.\label{obs:realizing_2_subfunctional:item:progress}\label{obs:realizing_2_subfunctional:item:last}
    \end{enumerateNumO}
\end{obs}
\begin{proof}
    Since $\depth_\calF(g) \geq 4$, if for some $i\in [d]$ and $j \in [d_i]$  item~\ref{obs:realizing_2_subfunctional:item:depth} holds, then item~\ref{obs:realizing_2_subfunctional:item:not_1} holds.
    Let $I$ be the set of all pairs of integers $i \in [d], j \in [d_i]$ such that item~\ref{obs:realizing_2_subfunctional:item:depth} is satisfied. By definition, $I$ is nonempty. Fix some $(i,j) \in I$, and let $t := t_j^{(i)}$, $x' := x|t_j^{(i)}$. We claim that $t(x') = \|x'\|_\calF$. Indeed, if there exists $t' \in W(\calF)$ with $t(x') < t'(x')$, then $g'$ obtained from $g$ by replacing $t$ with $t'$ satisfies $\|x\|_\calF = g(x) < g'(x)$, which is a contradiction.

    Suppose that for every $(i,j) \in I$, the functional $t_j^{(i)}$ is not $\calF$-realizing for $x|t_j^{(i)}$. 
    That is, $\depth_\calF(t_j^{(i)}) > j(x|t_j^{(i)})$. 
    For each $(i,j) \in I$, let $s_j^{(i)}$ be $\calF$-realizing for $x|t_j^{(i)}$. 
    Note that $\depth_\calF(t_j^{(i)}) > \depth_\calF(s_j^{(i)})$. 
    Let $g'$ be obtained from $g$ by replacing $t_j^{(i)}$ with $s_j^{(i)}$ for each $(i,j) \in I$. 
    It follows that $g(x) = g(x')$, $\depth_\calF(g) > \depth_\calF(g')$, and $\spn g' \subset \spn g$, which contradicts the fact that $g$ is $\calF$-realizing for $x$. 
    Therefore, there exists $(i,j) \in I$ such that $t_j^{(i)}$ is $\calF$-realizing for $x|t_j^{(i)}$.

    Finally, we prove that for $(i,j) \in I$ as above item~\ref{obs:realizing_2_subfunctional:item:progress} hold. Observe that $t_j^{(i)}(x) < g(x) = \|x\|_\calF$, as otherwise, $g$ is not $\calF$-realizing. It follows that $\spn t_j^{(i)}$ is a strict subset of $\spn g$, and so $\spn x$. 
\end{proof}

\begin{lemma}\label{lem:realizing:sequence}
    Let $\calF$ be a regular family. For every $x \in c_{00}^+$, there exist $c \in \N_0$ and $f_0,\dots,f_c \in W(\calF)$ such that
    \begin{enumerateNumR}
        \item $\depth_\calF(f_0) \leq 3$;\label{lem:realizing:sequence:item:first}\label{lem:realizing:sequence:item:f_0}
        \item $f_{c}$ is $\calF$-realizing for $x$;\label{lem:realizing:sequence:item:realizing:f_c}
        \item for every $m \in [c]$, $f_{m-1}$ is $\calF$-realizing for $x|f_{m-1}$;\label{lem:realizing:sequence:item:realizing}
        \item for every $m \in [c]$, $\spn f_{m-1} \subset \spn f_m$;\label{lem:realizing:sequence:item:supp:inclusion}
        \item for every $m \in [c]$, $\depth_\calF(f_m) = \depth_\calF(f_{m-1}) + 2 $;\label{lem:realizing:sequence:item:depth}
        \item for every $m \in [c]$, if $\calF$ has the insertion property, then  $\min \supp f_m < \min \supp f_{m-1}$;\label{lem:realizing:sequence:item:progress}
        \item for every $m \in [c]$, there exist $F_1,F_2 \subset \N$ with $\min \supp f_m \in F_1$ and
            \[\min \supp f_m \leq F_1 \leq \spn f_{m-1} < F_2 \leq \max \supp f_m,\]
        such that $F_1 \cup F_2 \in \full_{f_m}(\calF)$. \label{lem:realizing:sequence:item:fullness:cor}\label{lem:realizing:sequence:item:last}
    \end{enumerateNumR} 
\end{lemma}
\begin{proof}
    Suppose the lemma was false. 
    Let $x \in c_{00}^+$ be a counterexample with $|\supp x|$ is minimal. 
    Let $f \in W(\calF)$ be an $\calF$-realizing functional for $x$. 
    By \cref{lem:force:star}, there exists $g \in W(\calF)$ that is not $\calF$-worse than $f$ for $x$ and one of either \ref{lem:force:star:depth} or \ref{lem:force:star:full} is satisfied. 
    In the case, where $\calF$ has the insertion property, by \cref{lem:upper:main}, the condition~\ref{lem:force:star:full} can be replaced with \ref{lem:upper:main:iii}. 
    It follows that $g$ is $\calF$-realizing for $x$.
    If $\depth_\calF(g) \leq 3$, then we put $c:=0$ and $f_0:=g$. 
    It is easy to check that items~\ref{lem:realizing:sequence:item:first}-\ref{lem:realizing:sequence:item:last} are satisfied. 

    Therefore, we can assume that~\ref{lem:force:star:full} (or \ref{lem:upper:main:iii}) is satisfied, and $\depth_\calF(g) \geq 4$. 
    Let $d \in \N$ and let $g_1,\dots,g_d \in W(\calF)$ be such that $(g_1,\dots,g_d)$ is $\calF$-full-building for $g$. For each $i \in [d]$ with $\depth_\calF(g_i) = 0$, let $d_i := 0$; for each $i \in [d]$ with $\depth_\calF(g_i) \geq 1$ let $d_i \in \N$ be such that there exist $t_1^{(i)},\dots,t_{d_i}^{(i)} \in W(\calF)$ with $(t_1^{(i)},\dots,t_{d_i}^{(i)})$ being $\calF$-building for $g_i$.
    In the case, where $\calF$ has the insertion property, by \ref{lem:upper:main:iii} we can assume that either $\depth_\calF(g_1) = 0$ or $\depth_\calF(t_1^{(1)}) \leq 1$.
    By \cref{obs:realizing_2_subfunctional}, there exist $i \in [d]$ with $\depth_\calF(g_i) \geq 1$ and $j \in [d_i]$ such that \ref{obs:realizing_2_subfunctional:item:first}-\ref{obs:realizing_2_subfunctional:item:last} are satisfied.
    Let $t := t_j^{(i)}$.
    By minimality of $x$ and~\ref{obs:realizing_2_subfunctional:item:progress}, for $x|t$ there exist $c'\in \N$ and $f_0,\dots,f_{c'} \in W(\calF)$ such that \ref{lem:realizing:sequence:item:first}-\ref{lem:realizing:sequence:item:last} hold.
    From now on, we refer to the statements in the items for $x|t$ and $f_0,\dots,f_{c'}$ as [\ref{lem:realizing:sequence:item:first}]-[\ref{lem:realizing:sequence:item:last}]. 
    Let $c := c' + 1$, and let $f_c := g$. We claim that the sequence  $f_0,\dots,f_c$ satisfy \ref{lem:realizing:sequence:item:first}-\ref{lem:realizing:sequence:item:last} for $x$. Since $x$ is a counterexample this claim leads to a contradiction, thus, it suffices to end the proof of the lemma.

    Items~\ref{lem:realizing:sequence:item:f_0}~and~\ref{lem:realizing:sequence:item:realizing:f_c} are obvious. 
    Note that for each $m \in [c-1] = [c']$, the statements in \ref{lem:realizing:sequence:item:realizing}-\ref{lem:realizing:sequence:item:fullness:cor} follow from the corresponding statements in [\ref{lem:realizing:sequence:item:realizing}]-[\ref{lem:realizing:sequence:item:fullness:cor}], hence it suffices to prove them for $m=c$. 
    Item~\ref{lem:realizing:sequence:item:realizing} follows from [\ref{lem:realizing:sequence:item:realizing:f_c}].
    By~[\ref{lem:realizing:sequence:item:realizing:f_c}], we have 
        \[\spn f_{c-1} \subset \spn x|t = \spn t \subset \spn g = \spn f_c.\]
    This yields \ref{lem:realizing:sequence:item:supp:inclusion}. 
    By~\ref{obs:realizing_2_subfunctional:item:depth}~and~\ref{obs:realizing_2_subfunctional:item:realizing}, we have
        \[\depth_\calF(f_c) = \depth_\calF(g) = \depth_\calF(t) + 2 = \depth_\calF(f_{c-1}) + 2.\] 
    This yields \ref{lem:realizing:sequence:item:depth}.
    Recall that $t=t_j^{(i)}$.
    To prove the next item, assume that $\calF$ has the insertion property, by~\ref{obs:realizing_2_subfunctional:item:not_1}, we have $(i,j) \neq (1,1)$, therefore, 
        \[\min \supp f_c = \min \supp g \leq \min \supp g_1 < \min \supp t \leq \min \supp f_{c-1}.\]
    This yields~\ref{lem:realizing:sequence:item:progress}.
    For the last item we define
            \begin{align*}
                F_1 &= \{\min \supp g_\ell : \ell \in [i]\},\\
                F_2 &= \{\min \supp g_\ell : \ell \in [i+1,d]\}.
            \end{align*}
    Since $(g_1,\dots,g_d)$ is $\calF$-full-building for $g$, we have $F_1 \cup F_2 \in \full_{g}(\calF) = \full_{f_c}(\calF)$. The sequence of inequalities in~\ref{lem:realizing:sequence:item:fullness:cor} is clear, thus,~\ref{lem:realizing:sequence:item:fullness:cor} follows. 
\end{proof}

        \section{Upper bounds}\label{sec:upper}

	\subsection{Upper bound for $\calS_\varphi$}
    Fix an increasing and superadditive function $\varphi: \N \rightarrow \N$.

\begin{lemma} \label{lem:recursive}
Let $a,b \in \N$ with $a \leq b$.
    If $|[a,b]|\leq \varphi(a)$, then $j_{\Sf}(a,b) \leq 1$. Otherwise,
    \[j_{\Sf}(a,b) \leq 2 + \max\left(\{j_{\Sf}(a+i,b-\varphi(a)+i+1) : i \in [\varphi(a)-1]\} \cup \{1\}\right).\]
\end{lemma}
\begin{proof}
    The first part follows immediately from \cref{lem:force:F}. Suppose that $|[a,b]| > \varphi(a)$. 
    Let $x \in c_{00}^+$ be such that $j_{\Sf}(x) = j_{\Sf}(a,b)$. 
    By \cref{lem:realizing:sequence}, there exist $c \in \N_0$ and $f_0,\dots,f_c \in W(\Sf)$ such that~\ref{lem:realizing:sequence:item:first}-\ref{lem:realizing:sequence:item:last} are satisfied. By \cref{lem:S-y-nice}, the family $\Sf$ has the insertion property, thus, by~\ref{lem:realizing:sequence:item:progress}, we have $\min \supp f_c < \min \supp f_{c-1}$.
    
    By~\ref{lem:realizing:sequence:item:realizing:f_c}, $f_c$ is $\calF$-realizing for $x$.
    If $c=0$, then by~\ref{lem:realizing:sequence:item:f_0}, we have 
        \[j_{\Sf}(a,b) = j_{\Sf}(x) = \depth_{\Sf}(f_0) \leq 3 = 2 + 1.\]
    Suppose that $c \geq 1$.
    First, we claim that there exists $i \in [\varphi(a)-1]$ such that $\spn f_{c-1} \subset [a+i,b-\varphi(a)+i+1]$.
    By~\ref{lem:realizing:sequence:item:fullness:cor}, there exist $F_1,F_2 \subset \N$ with
            \[\min \supp f_c \leq F_1 \leq \spn f_{c-1} < F_2 \leq \max \supp f_c,\]
    such that $F_1 \cup F_2 \in \full_{f_c}(\Sf)$. Since $\depth_{\Sf}(f_c) \geq 2$, it follows that $\spn f_c \notin \Sf$, and so, \cref{lem:full:Sf} gives $|F_1 \cup F_2| = \varphi(\min F_1) \geq \varphi(a)$. Let $e := |F_1|$, we have 
        \[\spn f_{c-1} \subset [\max F_1,\min F_2 - 1] \subset [a+(e-1),b-(\varphi(a)-e)].\]
    If $2 \leq e \leq \varphi(a)$, then we put $i := e-1$, and in turn, $\spn f_{c-1} \subset [a+i,b-\varphi(a)+i+1]$. If $\varphi(a) < e$, then we put $i:=\varphi(a)-1$, and we have $\spn f_{c-1} \subset [a+(e-1),b]\subset [a+(\varphi(a)-1),b]$. Suppose that $e = 1$, it follows that $\spn f_{c-1} \subset [a,b-\varphi(a)+1]$. 
    However, $\min \supp f_c < \min \supp f_{c-1}$, thus, $\spn f_{c-1} \subset [a+1,b-\varphi(a)+1] \subset [a+1,b-\varphi(a)+1+1]$, and we put $i:=1$. This concludes the claim that there exists $i \in [\varphi(a)-1]$ such that $\spn f_{c-1} \subset [a+i,b-\varphi(a)+i+1]$.
    By~\ref{lem:realizing:sequence:item:depth},~\ref{lem:realizing:sequence:item:realizing},~and~\cref{obs:j-ab:ineq}, we have
        \begin{align*}
            j_{\Sf}(a,b) = j_{\Sf}(x) &= \depth_{\Sf}(f_c) \\
            &= \depth_{\Sf}(f_{c-1}) + 2 \\
            &= j_{\Sf}(x|f_{c-1}) + 2 \leq j_{\Sf}(\min \supp f_{c-1}, \max \supp f_{c-1}) + 2 \\
            &\leq j_{\Sf}(a+i,b-\varphi(a)+i+1) + 2.\qedhere
         \end{align*} 
\end{proof}

The above lemma justifies the following definition. For all $a,b \in \N$ with $a\leq b$ let
       \[\hj(a,b) := 
        \begin{cases}
         1 & \text{if } |[a,b]| \leq \varphi(a) , \\
         2 + \max\left(\{\hj(a+i,b-\varphi(a)+i+1) : i \in [\varphi(a)-1]\} \cup \{1\}\right) & \text{otherwise.}
        \end{cases} \]
Clearly, for all $a,b \in \N$, we have $j_{\Sf}(a,b) \leq \hj(a,b)$. Moreover, applying an elementary induction, we obtain an analogous monotonicity result to \cref{obs:j-ab:ineq} for $\hj$ -- see below.
\begin{obs}\label{obs:hj-subset-property}
    For all positive integers $a,b,c,d$ such that $[a,b]\subset [c,d]$ we have $\hj(a,b) \leq \hj(c,d)$.
\end{obs}

Furthermore, we can prove a stronger property for $\hj$ that is intuitive for $j_{\Sf}$ but apparently difficult and technical to derive.

\begin{lemma}\label{lem:hj-t-property}
    For all $a,b,t \in \N$, we have $\hj(a+t,b+t) \leq \hj(a,b)$.
\end{lemma}
\begin{proof}
    We proceed by induction on $s = b-a$. If $|[a,b]| \leq \varphi(a)$ then for every $t \in \N$, we have $1 = \hj(a+t,b+t) = \hj(a,b)$. Suppose that $|[a,b]| > \varphi(a)$. By induction,
        \[\hj(a,b) = 2 + \max\left(\{\hj(a+i,b-\varphi(a)+i+1) : i \in [\varphi(a)-1]\} \cup \{1\}\right) = 2 + \hj(a+1,b-\varphi(a)+2).\]
    Similarly, for every $t \in \N$,
        \[\hj(a+t,b+t) = 2 + \hj(a+t+1,b+t-\varphi(a+t)+2).\]
    Therefore, it suffices to prove that
        \[\hj(a+t+1,b+t-\varphi(a+t)+2) \leq \hj(a+1,b-\varphi(a)+2).\]
    By \cref{obs:hj-subset-property} and since $\varphi(a) < \varphi(a+t)$,
        \[\hj(a+t+1,b+t-\varphi(a+t)+2) \leq \hj(a+t+1,b+t-\varphi(a)+2).\]
    Finally, by induction,
        \[\hj(a+t+1,b+t-\varphi(a)+2) \leq \hj(a+1,b-\varphi(a)+2),\]
    which ends the proof.
\end{proof}

In particular, the above lemma gives that for all $a,b \in \N$,
       \[\hj(a,b) = 
        \begin{cases}
         1 & \text{if } |[a,b]| \leq \varphi(a) , \\
         2 + \hj(a+1,b-\varphi(a)+2) & \text{otherwise.}
        \end{cases} \]

\begin{lemma}\label{lem:upper-S-y-main}
    For all $a,b \in \N$ and every $c \in \N_0$ if
        \[|[a,b]| \leq \sum_{i=0}^{c} \varphi(a+i) - c,\]
    then
        \[\hj(a,b) \leq 2c+1.\]
\end{lemma}
\begin{proof}
    We proceed by induction on $c$. If $c=0$, then $|[a,b]| \leq \varphi(a)$, clearly yields $\hj(a,b) = 1$.
    Suppose that $c \geq 1$. If $|[a,we assumed thatb]| \leq \varphi(a)$, then the assertion follows, hence, let us assume otherwise. We have $\hj(a,b) \leq 2 + \hj(a+1,b-\varphi(a)+2)$. Observe that
        \[|[a+1,b-\varphi(a)+2]| = |[a,b]|-\varphi(a)+1 \leq \sum_{i=0}^{c} \varphi(a+i) - c - \varphi(a)+1 = \sum_{i=0}^{c-1} \varphi((a+1)+i) - (c-1).\]
    Therefore, by induction, $\hj(a+1,b-\varphi(a)+2) \leq 2(c-1)+1$, and so, $\hj(a,b) \leq 2c+1$.
\end{proof}

\begin{theorem}\label{th:upper:S-y}
    For every $n \in \N$ and every $c \in \N_0$ if
        \[n \leq \sum_{i=1}^{c+1} \varphi(i) - c\]
    then
        \[j_{\Sf}(n) \leq 2c+1.\]
\end{theorem}
\begin{proof}
    By \cref{lem:upper-S-y-main}, applied with $a=1$ and $b=n$, we obtain $\hj(1,n) \leq 2c+1$, and so,  $j_{\Sf}(n) = j_{\Sf}(1,n) \leq \hj(1,n)$.
\end{proof}


	\subsection{Upper bound for $k \calS_1$}
    Fix an integer $k$ with $k \geq 2$. As already mentioned the family $k \calS_1$ does not have the insertion property, hence, we need a different approach than in the case of $\Sf$.

\begin{theorem}\label{thm:upper:kS_1}
    For every $n \in \N$, we have
        \[j_{2\calS_1}(n) \leq 4 \log n + 25,\]
    and for every integer $k$ with $k \geq 3$ we have
        \[j_{k \calS_1}(n) \leq \frac{8}{k-2}\log n + 3.\]
\end{theorem}
\begin{proof}
    Most of the proof is the same for the case of $k=2$ and $k \geq 3$, however, there is one detail that differs.
    Let us start by fixing some $n \in \N$ and an integer $k$ with $k \geq 2$.
    Next, let us fix $x \in c_{00}^+$ such that $j_{k\calS_1}(n) = j_{k\calS_1}(x)$. 
    We apply \cref{lem:realizing:sequence} to the family $k\calS_1$ and $x$ to obtain $c \in \N_0$ and $f_0,\dots,f_c \in W(k \calS_1)$ such that~\ref{lem:realizing:sequence:item:first}-\ref{lem:realizing:sequence:item:last} hold.
    Note that by~\ref{lem:realizing:sequence:item:realizing:f_c},~\ref{lem:realizing:sequence:item:f_0},~and~\ref{lem:realizing:sequence:item:depth},
        \[j_{k\calS_1}(n) = j_{k\calS_1}(x) = \depth_{k\calS_1}(f_c) \leq 2c+3.\]
    In the case, where $c=0$, this concludes the proof, so assume that $c \geq  1$.
    Let $k' := \lfloor k \slash 2 \rfloor$ and define
        \[Z := \{m \in [c] : \min \supp f_{m-1} \geq 2^{k'}\min \supp f_{m}\}. \]
    Let $z := |Z|$. Clearly,
    \begin{align*}\label{eq:kS:z}
        n \geq \min \supp f_0 \geq 2^{k'z} \min \supp f_c \geq 2^{k'z}.
    \end{align*}
    Let $Z' := [c] \backslash Z$ and $z' := |Z'|$. 
    Fix some $m \in Z'$.
    Let $F_1$ and $F_2$ be as in~\ref{lem:realizing:sequence:item:fullness:cor}. 
    We have, 
        \[\min F_1 = \min \supp f_m \leq \max F_1 \leq \min \supp f_{m-1} < 2^{k'}\min \supp f_{m}.\]
    Therefore, and by \cref{lem:full-sets-sequence-in-S_2}, $E_{i}(F_1 \cup F_2) \subset F_2$ for every integer $i$ with $i \geq k'+1$. 
    Since $F_1 \cup F_2 \in \full_{f_m}(k\calS_1)$ and $\depth_{k\calS_1}(f_m) \geq 2$, we have $\spn f_m \notin k\calS_1$.
    It follows that the assumptions of \cref{lem:full-sets-in-kS_1} are satisfied, and so, $E_1(F_1\cup F_2), \dots, E_{k-1}(F_1\cup F_2)$ are full Schreier sets.
    First, consider the case, where $k \geq 3$.
    For each $i \in [k'+1,k-1]$, the set $E_i(F_1 \cup F_2)$ is a full Schreier set and is a subset of $F_2$.
    What is more, $F_2 \subset [\max \supp f_{m-1}+1,\max \supp f_m]$. By \cref{lem:full-sets-sequence-in-S_2},
        \begin{align*}\label{eq:kgeq3}
            \max \supp f_m \geq 2^{k-1-k'} (\max \supp f_{m-1} + 1) \geq 2^{k-1-k'} \max \supp f_{m-1}.
        \end{align*}
    Next, we focus on the case, where $k=2$. By \cref{lem:full-sets-in-kS}, 
        \[\max \supp f_m \slash 2 + 2 \geq \min E_2(F_1 \cup F_2) \geq \min F_2 > \max \supp f_{m-1}.\]
    In particular, $\max \supp f_m \geq 2 \max \supp f_{m-1} - 4$.
    Summing up, in the case of $k \geq 3$,
        \[ n \geq \max \supp f_c \geq 2^{(k-1-k')z'}\max \supp f_0 \geq 2^{(k-1-k')z'},\]
    and in the case of $k = 2$,
        \[n \geq \max \supp f_c \geq 2^{z'}\max \supp f_0 - 4z' \geq 2^{z'}-4z'.\]
    Combining this with the relation of $n$ and $z$, if $k \geq 3$, we have
        \[n \geq 2^{(k \slash 2 - 1) \max \{z,z'\}} \geq 2^{(k \slash 2 - 1) c \slash 2}.\]
    Finally,
        \[\frac{8}{k-2}\log n + 3 \geq 2c + 3 \geq j_{k \calS_1}(n).\]
    Similarly, in the case of $k=2$, for $c \geq 12$, we have
        \[n \geq 2^{\max \{z,z'\} - 4\max \{z,z'\}} \geq 2^{c \slash 2} - 2c \geq 2^{c \slash 2 - 1},\]
    and so, $4 \log n + 7 \geq 2c+3 \geq j_{k \calS_1}(n)$.
    If $c < 12$, then $j_{k \calS_1}(n) \leq 25$, thus, the bound in the assertion holds.
\end{proof}

	\subsection{Upper bound for $\calS_2$}
    The family $\calS_2$ has the insertion property, however, we will also use some ideas from the previous section. Note that the proof of the upper bound for $\calS_3$ is very similar.

\begin{theorem}\label{thm:upper:S2}
    For every $n \in \N$, we have
        \[j_{\calS_2}(n) \leq 8 \sqrt{\log n} + 9.\]
\end{theorem}
\begin{proof}
    Fix some $n \in \N$ and $x \in c_{00}^+$ such that $j_{\calS_2}(n) = j_{\calS_2}(x)$. 
    We apply \cref{lem:realizing:sequence} to the family $\calS_2$ and $x$ to obtain $c \in \N_0$ and $f_0,\dots,f_c \in W(\calS_2)$ such that~\ref{lem:realizing:sequence:item:first}-\ref{lem:realizing:sequence:item:last} hold.
    Note that by~\ref{lem:realizing:sequence:item:realizing:f_c},~\ref{lem:realizing:sequence:item:f_0},~and~\ref{lem:realizing:sequence:item:depth} we have 
        \[j_{\calS_2}(n) = j_{\calS_2}(x) = \depth_{\calS_2}(f_c) \leq 2c+3.\]
    In the case, where $c=0$, this concludes the proof, so assume that $c \geq  1$.
    Define
        \[Z := \{m \in [c] : \min \supp f_{m-1} \geq \min \supp f_{m} \cdot 2^{\min \supp f_{m} \slash 2}\}. \]
    Let $z := |Z|$. By \cref{lem:S-2-nice} and~\ref{lem:realizing:sequence:item:progress}, we have $\min \supp f_m \geq c-m+1$. By~\ref{lem:realizing:sequence:item:supp:inclusion}, we obtain
        \[n \geq \min \supp f_0 \geq \prod_{i=c-z+1}^c 2^{\min \supp f_{i} \slash 2} \min \supp f_c \geq \prod_{i=1}^z 2^{i\slash 2} \geq 2^{(z^2+z)\slash 4}.\]
    Let $Z' := [c] \backslash Z$ and $z' := |Z'|$. 
    Fix some $m \in Z'$.
    Let $F_1$ and $F_2$ be as in~\ref{lem:realizing:sequence:item:fullness:cor}, and let $a := \min F_1$.
    We have, 
        \[a = \min F_1 = \min \supp f_m \leq \max F_1 \leq \min \supp f_{m-1} < \min \supp f_{m} \cdot 2^{\min \supp f_{m} \slash 2} = 2^{a \slash 2} a.\]
    In particular, $F_1 \subset [a,2^{a \slash 2} a]$. 
    Let $s$ be the greatest positive integer such that $E_s(F_1 \cup F_2) \subset F_1$, or let $s := 0$ if $E_1(F_1 \cup F_2)\not\subset F_1$.
    By \cref{lem:full-sets-sequence-in-S_2}, $2^{a \slash 2} a \geq 2^s a$, and so $a \slash 2 \geq s$.
    It follows that $E_{i}(F_1 \cup F_2) \subset F_2$ for every integer $i$ with $i > a \slash 2$. 
    Since $F_1 \cup F_2 \in \full_{f_m}(\calS_2)$ and $\depth_{\calS_2}(f_m) \geq 2$, we have $\spn f_m \notin \calS_2$.
    This ensures that the assumptions of \cref{lem:full-sets-in-S_2} are satisfied, and so $E_1(F_1\cup F_2), \dots, E_{a-1}(F_1\cup F_2)$ are full Schreier sets.
    For each $a \slash 2 < i \leq a$, the set $E_i(F_1 \cup F_2)$ is a full Schreier set and is a subset of $F_2$.
    What is more, $F_2 \subset [\max \supp f_{m-1}+1,\max \supp f_m]$. By \cref{lem:full-sets-sequence-in-S_2},
        \begin{align*}
            \max \supp f_m \geq 2^{a - (a \slash 2 + 1)} (\max \supp f_{m-1} + 1) \geq 2^{a \slash 2 - 1} \max \supp f_{m-1}.
        \end{align*}
    Therefore,
        \[ n \geq \max \supp f_c \geq \prod_{i=c-z+1}^c 2^{\min \supp f_{i} \slash 2 - 1}\max \supp f_0 \geq \prod_{i=1}^{z'}2^{i \slash 2-1} = 2^{(z'^2-z')\slash 4}.\]
    Combining this with the relation of $n$ and $z$ and using $z+z' = c$, we obtain
        \[n \geq \max\{2^{(z^2+z)\slash 4}, 2^{(z'^2-z')\slash 4}\}  \geq 2^{(c^2-6c)\slash 16} \geq 2^{(c-3)^2\slash 16}.\]
    Finally,
        \[8\sqrt{\log n} + 9 \geq 2c + 3 \geq j_{\calS_2}(n).\qedhere\]
\end{proof}

	\subsection{Upper bound for $\calS_3$}
    \begin{theorem}\label{thm:upper:S3}
    For every $n \in \N$ we have
        \[j_{\calS_3}(n) \leq 8 \sqrt{\log^* n} + 9.\]
\end{theorem}
\begin{proof}
    Fix some $n \in \N$ and $x \in c_{00}^+$ such that $j_{\calS_3}(n) = j_{\calS_3}(x)$. 
    We apply \cref{lem:realizing:sequence} to the family $\calS_3$ and $x$ to obtain $c \in \N_0$ and $f_0,\dots,f_c \in W(\calS_3)$ such that~\ref{lem:realizing:sequence:item:first}-\ref{lem:realizing:sequence:item:last} hold.
    Note that by~\ref{lem:realizing:sequence:item:realizing:f_c},~\ref{lem:realizing:sequence:item:f_0},~and~\ref{lem:realizing:sequence:item:depth} we have 
        \[j_{\calS_3}(n) = j_{\calS_3}(x) = \depth_{\calS_3}(f_c) \leq 2c+3.\]
    In the case, where $c=0$, this concludes the proof, so assume that $c > 0$.
    Define
        \[Z := \{m \in [c] : \min \supp f_{m-1} \geq \tau(\lfloor\min \supp f_{m} \slash 2\rfloor, \min \supp f_{m})\}. \]
    Let $z := |Z|$. By \cref{lem:S-2-nice} and~\ref{lem:realizing:sequence:item:progress} we have $\min \supp f_m \geq c-m+1$. By~\ref{lem:realizing:sequence:item:supp:inclusion}, we obtain
    \begin{align*}
        n \geq \min \supp f_0 \geq \tau(\sum_{i=c-z+1}^c \lfloor\min \supp f_{i} \slash 2\rfloor, &\min \supp f_c) \geq \\
        &\tau(\sum_{i=1}^z \lfloor i \slash 2\rfloor, 1) \geq \tau((z^2-3z)\slash 4,1).
    \end{align*}
        \[\]
    Let $Z' := [c] \backslash Z$ and $z' := |Z'|$. 
    Fix some $m \in Z'$.
    Let $F_1$ and $F_2$ be as in~\ref{lem:realizing:sequence:item:fullness:cor}, and let $a := \min F_1$.
    We have, 
        \begin{align*}
            a = \min F_1 = \min \supp f_m &\leq \max F_1 \leq \min \supp f_{m-1}\\ &< \tau(\lfloor\min \supp f_{m} \slash 2\rfloor, \min \supp f_{m}) = \tau(\lfloor a \slash 2\rfloor, a).
        \end{align*}
    In particular, $F_1 \subset [a,\tau(\lfloor a \slash 2\rfloor, a)]$. 
    Let $s$ be the greatest positive integer such that $E_s(F_1 \cup F_2) \subset F_1$, or let $s := 0$ if $E_1(F_1 \cup F_2)\not\subset F_1$.
    By \cref{lem:full-sets-sequence-in-S_3}, $\tau(\lfloor a \slash 2\rfloor, a) \geq \tau(s,a)$, and so, $\lfloor a \slash 2\rfloor \geq s$.
    It follows that $E_{i}(F_1 \cup F_2) \subset F_2$ for every integer $i$ with $i > \lfloor a \slash 2\rfloor$. 
    Since $F_1 \cup F_2 \in \full_{f_m}(\calS_3)$ and $\depth_{\calS_3}(f_m) \geq 2$, we have $\spn f_m \notin \calS_3$.
    This ensures that the assumptions of \cref{lem:full-sets-in-S_3} are satisfied, and so, $E_1(F_1\cup F_2), \dots, E_{a-1}(F_1\cup F_2)$ are $\calS_2$-full sets.
    For each $\lfloor a \slash 2\rfloor < i \leq a$, the set $E_i(F_1 \cup F_2)$ is an $\calS_2$-full set and is a subset of $F_2$.
    What is more, $F_2 \subset [\max \supp f_{m-1}+1,\max \supp f_m]$. By \cref{lem:full-sets-sequence-in-S_3},
        \begin{align*}
            \max \supp f_m \geq \tau(a - (\lfloor a \slash 2\rfloor + 1), \max \supp f_{m-1} + 1) \geq \tau(a \slash 2-1, \max \supp f_{m-1} ).
        \end{align*}
    Therefore,
    \begin{align*}
        n \geq \max \supp f_c \geq \tau(\sum_{i=c-z+1}^c \min \supp f_{i} \slash &2-1, \max \supp f_0) \geq \\
        &\tau(\sum_{i=1}^{z'} i \slash 2-1, 1) = \tau((z'^2-3z')\slash 4,1).
    \end{align*}
    Combining this with the relation of $n$ and $z$ and using $z+z' = c$, we obtain
        \[n \geq \max\{\tau((z^2-3z)\slash 2,1), \tau((z'^2-3z')\slash 4,1)\}  \geq  \tau((c^2-6c)\slash 16,1) \geq \tau((c-3)^2\slash 16,1).\]
    Finally,
        \[8\sqrt{\log^* n} + 9 \geq 2c + 3 \geq j_{\calS_3}(n).\qedhere\]
\end{proof}
	
	\section{Open problems}\label{sec:open}
    To conclude, we want to mention a few interesting research directions related to computing the function $j(n)$. The first natural problem is to give an even more precise estimation of the original function $j_{\calS_1}(n)$, that is, up to an additive constant.
\begin{problem}
    Find a real number $C$ such that there exist real numbers $A,B$ such that for every $n \in \N$,
        \[ C \sqrt{n} + A \leq j(n) \leq C \sqrt{n} + B.\]
\end{problem}
By \cref{th:main-Schreier}, we know that if the constant $C$ exists, then $C \in [\sqrt{2},2]$. 

Recall that a generalized Tsirelson's norm $T[\theta,\calF]$ depends on a real number $0 < \theta < 1$ and a regular family $\calF$. In this paper, we studied the function $j_{T[\frac{1}{2},\calF]}$ for some regular families $\calF$. Another approach is to determine what is the order of magnitude of the function $j_{T[\theta,\calF]}$ when we fix $\calF = \calS_1$ and change $\theta$. There are two versions of this problem.
\begin{problem}
    For a fixed real number $\theta$ with $0 < \theta < 1$, compute the order of magnitude of the function $j_{T[\theta,\calS_1]}(n)$.
\end{problem}
\begin{problem}
    Compute the order of magnitude of a two-variable function $j_{T[\theta,\calS_1]}(n)$.
\end{problem}

The last problem is inspired by the discussion in \cite{Go-blog}.
We believe that the problem of computing the Tsirelson's norm $\|x\|_T$ for a vector $x \in c_{00}$ with $\supp x \subset [n]$ can be solved using a dynamic programming scheme in polynomial time.
More precisely, in time $\mathrm{poly}(n) \cdot j(n)$, which is clearly a polynomial function.
The situation seems to be similar in the case of $\|x\|_{\Sf}$ for any function $\varphi$.
In particular, this gives a polynomial time algorithm for any fast-growing function $\varphi$.
However, for slow-growing functions, it does not give satisfactory running time -- recall that we established a lower bound on $j_{\Sf}$, which in the case of slow-growing functions is a fast-growing function. 
\begin{problem}
    Is there a non-decreasing function $\varphi: \N \rightarrow \N$ such that the problem of computing the norm $\|\cdot\|_{T[\frac{1}{2},\Sf]}$ is hard in the sense of computational complexity?
\end{problem}

\bibliographystyle{abbrv}
\bibliography{bib_source}

\end{document}